\documentclass[11pt]{amsart}

\usepackage[utf8]{inputenc}
\usepackage{amssymb}
\usepackage{amsmath}
\usepackage{amsthm}

\usepackage{epsfig}
\usepackage{color}
\usepackage{verbatim}
\usepackage{paralist}

\usepackage{tikz}

\usepackage{comment}
\usepackage{amsmath}
\usepackage{hyperref}
\usepackage{xypic}
\usepackage{amscd}
\usepackage{color}

\newfont{\sheaf}{eusm10 scaled\magstep1}

\newtheorem{thm}{Theorem}[section]

\newtheorem{lemma}[thm]{Lemma}

\newtheorem{proposition}[thm]{Proposition}

\theoremstyle{definition}
\newtheorem{remark}[thm]{Remark}

 \newtheorem{example}[thm]{Example}

\DeclareMathOperator{\Pic}{Pic}

\def\c1{\operatorname{c_1}}
\def\c2{\operatorname{c_2}}

\def\c{\mathbf{c}}

\def\PP{{\mathbb P}}

\def\cong{\simeq}

\def\+{\oplus}                   
\def\*{\otimes}                  

\def\Pic{\operatorname{Pic}}

\def\geq{\geqslant}
\def\leq{\leqslant}

\makeatletter
\@namedef{subjclassname@2020}{\textup{2020} Mathematics Subject Classification}
\makeatother

\begin{document}

\title[2--elementary rational covers of the plane]{Two--elementary rational covers of the plane}

\author[C.~Ciliberto]{Ciro Ciliberto}
\address{C.~Ciliberto, Dipartimento di Matematica,  Universit\`a di Roma ``Tor Vergata'', Via della Ricerca Scientifica, 00173 Roma, Italy}
\email{cilibert@axp.mat.uniroma2.it}

\begin{abstract} In this paper I classify, up to Cremona transformations, the Galois cover of the plane with Galois group of the form $\mathbb Z_2^r$. 
  \end{abstract}

\maketitle

\tableofcontents

\section*{Introduction} The study of finite, normal multiple covers of the plane is a classical subject in algebraic geometry. In particular a problem that is still widely open is the one of classifying, up to Cremona transformations of the plane, the multiple covers $f: S\longrightarrow \PP^2$ as above such that $S$ is rational. The first important contribution  to this subject goes back to the paper \cite {CE}  of 1900 by G. Castelnuovo and F. Enriques. In that paper the authors claimed to give the full classification of rational double covers of the plane up to Cremona transformations. The arguments in that paper were not complete  and a full treatment of the subject has been given only more recently by L. Bayle and A. Beauville \cite {BB} in 2000 and by A. Calabri in his PhD thesis published in 2006 (see \cite {Cal}), that is more in the style of the original Castenuovo--Enriques' approach. The result is stated here in Theorem \ref {thm:cal}. In Calabri's thesis there is also the full classification of rational, cyclic triple covers of the plane. 

Besides these results essentially nothing is known on this subject, for instance it is still unknown a classification, up to Cremona transformation of the plane, of rational non--Galois triple planes. In general the problem of classifying rational multiple covers of the plane seems to be extremely hard, probably out of our present possibilities. Therefore it seems reasonable to restrict the attention to some interesting particular cases. For instance I propose to focus on  Galois covers of the plane whose Galois group is finite and abelian. The reason for looking at these cases is the following. On one side J. Blanc, in his  thesis  \cite {Bl} of 2006, has proposed a classification of the finite abelian subgroups of the Cremona group of the plane, up to conjugation. So, in principle, we know for all abelian Galois covers $f: S\longrightarrow \PP^2$ with $S$ rational, what are the possible Galois group and how they acts on $S$, up to conjugation. On the other hand there is also the paper \cite {Par} of 1991 by R. Pardini, in which there is a detailed description of structure, in particular of the branch curves, of  finite, abelian covers of algebraic varieties. The idea is that a classification, up to Cremona transformations, of the branch curves of the rational abelian covers of the plane should be in principle possible, by putting together Blanc's and Pardini's results. 

Though in principle the above project looks feasible, in practice it is rather hard to be carried out, due to the several cases that Blanc's classification presents and due to the great variety of Galois groups entering into the scene, and this makes the Pardini's style analysis quite complicated and cumbersome. However, I think this is a project that is worth to be pursued. 

The present paper is a first contribution in this direction. In this paper in fact I give the classification of the branch curves of all rational abelian covers of $\PP^2$, up to Cremona transformations, with two--elementary Galois group, i.e., with Galois group of the type $\mathbb Z_2^r$. By Blanc's classification, it turns out that $1\leq r \leq 4$. As I said already, the case $r=2$ is classical and well known, so we can focus to the cases $2\leq r\leq 4$. 

According to Blanc's analysis, that I recall in Section \ref {ssec:cases} (see in particular Proposition \ref {prop:bla}), the classification is naturally divided in two main cases: the case in which there is a pencil of rational curves fixed by the group action (the \emph{invariant conic bundle case}) and the case in which the invariant part of the Picard group under the group action has rank 1, generated by the canonical bundle (the \emph{Del Pezzo case}). 

The classification in the invariant conic bundle case is contained in Propositions \ref {prop:G'=0}, \ref {prop:G'2}, \ref {prop:G'4} for $r=2$, in  Propositions \ref {prop:G'=0,3}, \ref {prop:G'=2,3} for $r=3$ and in  Proposition \ref {prop:G'=2,4} for $r=4$. 

The classification in the Del Pezzo case is contained in  Propositions \ref {prop:bert}, \ref {prop:geiser} for $r=2$, in  Propositions \ref {prop:geiserol}, \ref {prop:geisero} for $r=3$ and in Proposition 
\ref  {prop:r4} for $r=4$. 

The paper is organised as follows. In Section \ref {sec:prel} there are some generalities and preliminaries, in particular I mention the classification of rational double planes and Blanc's subdivision in the invariant conic bundle case and in the Del Pezzo case. In Section \ref {sec:blanc} I recall Blanc's classification of $\mathbb Z_2^r$ subgroups of the Cremona group of the plane up to conjugation. In particular, as I said, one has $1\leq r\leq 4$.
In Section \ref {sec:pard}, following Pardini's paper \cite {Par}, I recall the theory of $\mathbb Z_2^r$ covers of the plane, especially focusing on the cases $2\leq r\leq 4$ that are the ones we are interested in by Blanc's classification. Section \ref {sec:classinv} is devoted to the classification in the invariant conic bundle case, whereas Section \ref {sec:classinva} contains the classification in the Del Pezzo case.\medskip

\noindent {\bf Acknowledgements:} The author is a member of GNSAGA of the Istituto Nazionale di Alta Matematica. \medskip

\section{Preliminaries}\label{sec:prel}

\subsection{Generalities}  Let $S$ be an irreducible, complex, normal, projective surface, $G$ a finite group (that in this paper will often be assumed to be abelian) of automorphisms of $S$ and a finite morphism $f: S\longrightarrow \PP^2$ such that if $x\in \PP^2$ is a general point, $f^{-1}(x)$ is a full orbit of the $G$ action. In this case I will say that $f: S\longrightarrow \PP^2$, or the pair $(S,G)$,  is a \emph{$G$--cover} of $\PP^2$, and of course $\deg(f)=|G|$. 

Given a pair $(S,G)$ as above, that is also called a \emph{$G$--surface}, and given a birational morphism $\phi: S\longrightarrow S'$, with $S'$ an irreducible, normal, projective surface, one says that  $\phi$ is \emph{$G$--equivariant}, if the $G$--action on $S'$ induced by $\phi$ is biregular. The pair $(S,G)$ is said to be \emph{minimal} if any $G$--equivariant morphism $\phi: S\longrightarrow S'$ is an isomorphism. By a sequence of contractions of $(-1)$--curves, one can always achieve the situation that $(S,G)$ is minimal and therefore I will always assume  to deal with minimal pairs. 

Given an irreducible, smooth, projective surface $S$, and given a morphism $f: S\longrightarrow \PP^1$, one says that $(S,\phi)$ is a \emph{conic bundle structure} if:\\
\begin{inparaenum}
\item [$\bullet$] the general fibre $F$ of $\phi$ is smooth, irreducible and rational, so that $|F|$ is a pencil of rational curves;\\
\item [$\bullet$] all singular fibres of $\phi$ are of the form $F_1+F_2$ with $F_1,F_2$ smooth, irreducible and rational, with $F_1^2=F_2^2=-F_1\cdot F_2=-1$ (in particular $F_1,F_2$ are $(-1)$--curves). 
\end{inparaenum}

I will often say that the pencil $|F|$ is a \emph{conic bundle}.

\subsection{The Del Pezzo and the invariant conic bundle cases}\label{ssec:cases}

Given a $G$--surface $(S,G)$, with $S$ smooth, the group $G$ acts on $\Pic(S)$ and I will denote by $\Pic(S)^G$ the fixed part of $\Pic(S)$ by the $G$--action. I recall the following proposition from \cite [Prop. 2.3.1]{Bl}:

\begin{proposition}\label{prop:bla} Let $(S,G)$ be a minimal $G$--surface with $S$ smooth and rational. Then either:\\
\begin{inparaenum}
\item [(i)] $\Pic(S)^G$ has rank 1 and over $\mathbb Q$ it is generated by the canonical system, or\\
\item [(ii)] $\Pic(S)^G$ has rank 2, there a conic bundle structure $(S,\phi)$ on $S$ such that $\Pic(S)^G$ is generated over $\mathbb Q$ by the canonical system and by the general fibre $F$ of $\phi$. 
\end{inparaenum}
\end{proposition}

In case (i), $S$ is a Del Pezzo surface (see \cite [Lem. 4.1.5]{Bl}), i.e.,  $S$ is a rational surface whose anti-canonical divisor $-K_S$ is ample. Namely,
$S$ is either isomorphic to $\PP^2$, or to $\PP^1 \times \PP^1$ or to the blow-up of $r\leq 8$ points $p_1, p_2, \ldots, p_r$ in $\PP^2$ in sufficiently
general position. In this case I will say that we are in the \emph{Del Pezzo case}.

In  case (ii), there is a pencil $\phi: S\longrightarrow \PP^1$ with general fibre $F$ smooth, rational and irreducible. 
The group $G$ preserves this pencil. If $F$ is the generic fibre of $\phi$, then $G$ embeds into 
$${\rm Aut}(F)\rtimes {\rm PGL}(2, \mathbb C)\cong {\rm PGL}(2, \mathbb C(t))\rtimes {\rm PGL}(2, \mathbb C)$$
where $t$ is a variable.
 In this case I will say that $(S,g)$ is an \emph{invariant conic bundle structure} for $G$ and that we are in the \emph{invariant conic bundle} case.  

Note that if $(S,G)$ is a minimal $G$--surface with $S$ rational and normal, but not smooth, then $G$ acts also on the minimal desingularization $S'$ of $S$, and $(S',G)$ is minimal, so that Proposition \ref {prop:bla} can be applied to $(S',G)$. Then one says that $(S,G)$ presents the Del Pezzo or the invariant conic bundle structure if $(S',G)$ does. Note that if $(S,G)$ presents  the invariant conic bundle structure, there is a pencil $|F|$ of rational curves that are in general Weil divisors on $S$. 

Given a finite subgroup $G$ of the Cremona group of the plane, there is a (unique up to isomorphism) smooth rational surface $S$ and a birational map $\varphi: \PP^2\dasharrow S$ such that $\varphi \circ G\circ \varphi^{-1}\cong G$ acts biregularly on $S$ and $(S,G)$ is minimal (see \cite [Prop. 2.2.3]{Bl}, \cite [Thm. 1.4]{DE}). I will say that $G$ presents the Del Pezzo or the invariant conic bundle structure, if so is for $(S,G)$.

\subsection{Classification of rational double planes} 

In this paper I want to classify $G$--rational covers of $\PP^2$ up to Cremona transformations, in the case $G$ is of the form $\mathbb Z_2^r$. The case $r=2$ is classical, going back to G. Castelnuovo and F. Enriques \cite {CE}. In more recent times this case has been reconsidered by L. Bayle and A. Beauville  in \cite {BB}, and, in a way closer to Castelnuovo--Enriques' and our viewpoint here by A. Calabri in his Doctoral Thesis \cite {Cal}. The classification theorem  is as follows:

\begin{thm}\label{thm:cal} Let $f: S\longrightarrow \PP^2$ be a finite double cover of the plane with irreducible, projective, normal, surface. Then $S$ is rational if and only if, up to Cremona equivalence, the branch curve of $f$ is of one of the following types:\\
\begin{inparaenum}[(i)]
\item a reduced conic;\\
\item a smooth quartic curve;\\
\item an irreducible reduced  sextic curve with two infinitely near triple points;	\\
\item an irreducible reduced curve of degree $2d>2$ with an ordinary singularity of multiplicity $2d-2$ and either otherwise smooth or with only a further node. 
\end{inparaenum}
\end{thm}

Note that in case (i) if the conic is singular and in case (iv) we are in the invariant conic bundle case (on the minimal desingularization of $S$), whereas in cases (ii) and (iii) we are in the Del Pezzo case. In case 
(i), if the conic is irreducible, we are in the Del Pezzo case, since $S\cong \PP^1\times \PP^1$ and the double plane map induces the involution $(x,y)\in \PP^1\times \PP^1\longrightarrow (y,x)\in  \PP^1\times \PP^1$. In case (ii), $S$ is a smooth Del Pezzo surface with $K_S^2=2$, i.e., the plane blown--up at 7 sufficiently general points, and the double plane map induces an involution on $S$ called a \emph{Bertini involution}. In case (iii), the minimal desingularization of $S$ is a smooth Del Pezzo surface with $K_S^2=1$, i.e., the plane blown--up at 8 sufficiently general points, and the double plane map induces an involution on $S$ called a \emph{Geiser involution}. In this case, the surface $S$ can also be seen as the double cover of a quadric cone $Q$ in $\PP^3$ branched along a smooth curve cut out on $Q$ by a cubic hypersurface and at the vertex of the cone.   

In view of Theorem \ref {thm:cal}, in this paper I will disregard the case $G\cong \mathbb Z_2$ and I will focus on the cases $G\cong \mathbb Z_2^r$, with $r\geq 2$. I will base my analysis on two main fundamental results: (a) the classification in J. Blanc thesis \cite {Bl} of the abelian finite subgroups of the plane Cremona group; (b) R. Pardini's paper \cite {Par} in which there is a thorough analysis of the abelian covers  of algebraic varieties, which I will apply to $\mathbb Z_2^r$ rational covers of the plane. 

\section{Blanc's classification}\label{sec:blanc} In this section I recall Blanc's classification of the $\mathbb Z_2^r$--rational surfaces (with $r>1$), up to conjugation. I first start by considering  subgroups $G\cong \mathbb Z_2^r$, with $r>1$, of the Cremona group of the plane that present the invariant conic bundle case. 

\subsection{Preliminaries on the invariant conic bundle case}
If we have a $G$--surface $(S,G)$, and there is an invariant conic bundle on $S$, then  (on the minimal desingularization $S'$ of $S$) there is a base point free  pencil $|F|$ of rational curves on $S$ that is preserved by $G$, hence $G$ acts on the $\PP^1$ that parameterizes the pencil $|F|$, so that we have a homomorphism $\pi: G\longrightarrow {\rm PGL}(2,\mathbb C)$. Then we have an exact sequence
\begin{equation}\label{eq:split}
0\longrightarrow G'\longrightarrow G \stackrel{\pi }{\longrightarrow}\pi(G)\longrightarrow 0.
\end{equation}
Of course $G'$ is a subgroup of the automorphisms of the general fibre $F\in |F|$, that is  a $\PP^1$, so $G'$ is  also a subgroup of ${\rm PGL}(2,\mathbb C)$. 
In the situation under consideration in which $G\cong \mathbb Z_2^r$, with $r>1$, then both $G'$ and $\pi(G)$ are of the form $\mathbb Z_2^s$, with $0\leq s\leq r$. Moreover, by the classification of the finite abelian subgroups of ${\rm PGL}(2,\mathbb C)$, one has $s\leq 2$ and therefore $r\leq 4$. Finally, the sequence \eqref {eq:split} splits, because $G$ is also a finite dimensional vector space on $\mathbb Z_2$. 

Still in the  case that we are in the invariant conic bundle case, given $g\in G\setminus \{0\}$ such that $\pi(g)=0$, $g$ acts on every singular fibre of the conic bundle $|F|$ (still considered on the minimal desingularization $S'$ of $S$). If $F_1+F_2$ is such a fibre, 
and $g(F_1)=F_2$ (and consequently  $g(F_2)=F_1$), I will say that $g$ \emph{twists} the singular fibre $F_1+F_2$. I need to mention the following result (\cite [Lem. 7.1.3]{Bl}):

\begin{lemma}\label{lem:bl} Let  $(S,G)$ be a $G$--surface  with $S$ smooth, presenting the invariant conic bundle case. Consider a $g\in G\setminus \{0\}$ that twists at least one singular fibre of the conic bundle structure. Then the following are equivalent:\\
\begin{inparaenum}[(i)]
\item $g$ is an involution;\\
\item $\pi(g)=0$;\\
\item the set of points of $S$ fixed by $g$ is a smooth hyperelliptic curve of genus $k-1>0$, plus perhaps a finite number of
isolated points, which are the singular points of the singular fibres of $|F|$ not twisted by $g$;\\
\item $g$ fixes some curve of positive genus.
\end{inparaenum}

Moreover, if the above conditions are satisfied, the number of singular fibres of the conic bundle that are twisted by $g$ is $2k$.
\end{lemma}

In the setting of Lemma \ref {lem:bl}, if $g\in G$  twists at least one singular fibre of the conic bundle structure and one of the (and therefore all) conditions of Lemma \ref {lem:bl} are satisfied, one says that $g$ is a \emph{twisting De Jonqui\`ere involution}. 

In the situation under consideration in which $G\cong \mathbb Z_2^r$, with $r>1$, given a  non--zero element $g\in G$, either it does not twist any singular fibre of the conic bundle structure (hence it exchanges singular fibres) or,  being  an involution, one has $\pi(g)=0$. In other terms, the non--zero elements of $G'$ are exactly the twisting De Jonqui\`ere involutions in $G$. 

If $(S,G)$ is a $G$--surface presenting the invariant conic bundle case,   with $S$ normal but not smooth, one can apply Lemma \ref {lem:bl} to the minimal desingularization $S'$ of $S$, that is still a $G$--surface presenting the invariant conic bundle case. Then I can apply the above definition and terminology for $(S,G)$. 

\subsection{Blanc's classification in the invariant conic bundle structure case}

Here I state this theorem  (see \cite [Chapt. 10, Thm. B] {Bl}):

\begin{thm}\label{thm:class1} A subgroup $G\cong \mathbb Z_2^r$, with $r>1$, of the Cremona group that presents the invariant conic bundle case, is conjugate, in the Cremona group, to one (and only one) of the following\footnote{The symbols indicating the various cases are taken from Blanc's thesis.}:\\
\begin{inparaenum}
\item [{\bf (0.22)}] $r=2$, automorphisms of $\PP^1\times \PP^1$ preserving the pencil given by the first projection $p_1:  \PP^1\times \PP^1\longrightarrow \PP^1$, generated by $(x,y)\mapsto (-x,y)$ and $(x,y)\mapsto (x,-y)$, with $G'\cong \mathbb Z_2$;\\
\item [{\bf (P1.221)}] $r=2$, automorphisms of $\PP^1\times \PP^1$  preserving the pencil given by the first projection $p_1:  \PP^1\times \PP^1\longrightarrow \PP^1$, generated by $(x,y)\mapsto (x^{-1},y)$ and $(x,y)\mapsto (-x,-y)$, with $G'=0$;\\
\item [{\bf (P1.22.1)}] $r=2$, automorphisms of $\PP^1\times \PP^1$  preserving the pencil given by the first projection $p_1:  \PP^1\times \PP^1\longrightarrow \PP^1$, generated by $(x,y)\mapsto (\pm x^{\pm 1},y)$, with $G'=0$;\\
\item [{\bf (P1.2221)}] $r=3$, automorphisms of $\PP^1\times \PP^1$  preserving the pencil given by the first projection $p_1:  \PP^1\times \PP^1\longrightarrow \PP^1$, generated by $(x,y)\mapsto (\pm x^{\pm 1},y)$ and $(x,y)\mapsto (x,-y)$, with $G'\cong \mathbb Z_2$;\\
\item [{\bf (P1.222)}] $r=3$, automorphisms of $\PP^1\times \PP^1$  preserving the pencil given by the first projection $p_1:  \PP^1\times \PP^1\longrightarrow \PP^1$, generated by $(x,y)\mapsto (\pm x,\pm y)$ and $(x,y)\mapsto (x^{-1},y)$, with $G'\cong \mathbb Z_2$;\\
\item [{\bf (P1s.222)}] $r=3$ automorphisms of $\PP^1\times \PP^1$  preserving the pencil given by the first projection $p_1:  \PP^1\times \PP^1\longrightarrow \PP^1$, generated by $(x,y)\mapsto (x^{-1},y^{-1})$,  $(x,y)\mapsto (x,-y)$ and $(x,y)\mapsto (x,x/y)$, with $G'\cong \mathbb Z_2^2$;\\
\item [{\bf (P1.2222)}] $r=4$, automorphisms of $\PP^1\times \PP^1$  preserving the pencil given by the first projection $p_1:  \PP^1\times \PP^1\longrightarrow \PP^1$, generated by $(x,y)\mapsto (\pm x^{\pm 1},\pm y^{\pm 1})$, with $G'\cong \mathbb Z_2^2$;\\
\item [{\bf (C.2,21)}] $r=2$, birational maps  of $\PP^1\times \PP^1$ to itself preserving the pencil given by the first projection $p_1:  \PP^1\times \PP^1\longrightarrow \PP^1$, generated by 
\begin{align*}
&((x_1:x_2), (y_1:y_2))\dasharrow ((x_1:x_2), (y_2\prod_{i=1}^k(x_1-b_ix_2):y_1\prod_{i=1}^k(x_1-a_ix_2))),\\
&((x_1:x_2), (y_1:y_2))\mapsto((x_1:-x_2), (y_1:y_2))\\
\end{align*}
where $k$ is even,  $a_1,\ldots, a_k, b_1,\ldots, b_k\in \mathbb C^*$ are all distinct, and the sets $\{a_1,\ldots, a_k\}$ and
$\{b_1,\ldots, b_k\}$ are both invariant by multiplication by $-1$; here $G'\cong \mathbb Z_2$ contains exactly one twisting de Jonqui\`eres involution, that fixes
a hyperelliptic curve of genus $k-1$;\\
\item [{\bf (C.2,22)}] $r=3$, birational maps of $\PP^1\times \PP^1$ to itself preserving the pencil given by the first projection $p_1:  \PP^1\times \PP^1\longrightarrow \PP^1$, generated by 
\begin{align*}
&((x_1:x_2), (y_1:y_2))\dasharrow ((x_1:x_2), (y_2\prod_{i=1}^{k/4}P(b_i):y_1\prod_{i=1}^{k/4}P(a_i))),\\
&((x_1:x_2), (y_1:y_2))\mapsto((x_1:-x_2), (y_1:y_2))\\
&((x_1:x_2), (y_1:y_2))\mapsto((x_2:x_1), (y_1:y_2))\\
\end{align*}
where $k$ is divisible by 4, $P(z)=(x_1^2-z^2x_2^2)(x_1^2-z^{-2}x_2^2)$, 
$a_1,\ldots, a_{k/4}, b_1,\ldots, b_{k/4}\in \mathbb C\setminus \{0, \pm 1\}$ are all distinct; here $G'\cong \mathbb Z_2$ contains exactly one twisting de Jonqui\`eres involution, that fixes
a hyperelliptic curve of genus $k-1$;\\
\item [{\bf (C.22)}] $r=2$, $G=G'$ formed by birational maps of $\PP^1\times \PP^1$ to itself preserving the pencil given by the first projection $p_1:  \PP^1\times \PP^1\longrightarrow \PP^1$, generated by two twisting De Jonqui\'eres involutions
\begin{align*}
&(x,y)\dasharrow \Big (x, \frac {g(x)}y\	\Big ),\\
&(x,y)\dasharrow \Big (x, \frac {h(x)y - g(x)}{y - h(x)}\Big)\\
\end{align*}
for some $g(x), h(x) \in \mathbb C(x)\setminus \{0\}$, and where each of the two involutions fixes a hyperelliptic curve of positive genus;\\
\item [{\bf (C.221)
}] $r=3$, birational maps of $\PP^1\times \PP^1$ to itself preserving the pencil given by the first projection $p_1:  \PP^1\times \PP^1\longrightarrow \PP^1$, generated by 
\begin{align*}
&(x,y)\dasharrow \Big (x, \frac {g(x)}y\Big ),\\
&(x,y)\dasharrow \Big (x, \frac {h(x)y - g(x)}{y - h(x)}\Big)\\
&(x,y)\mapsto (-x,y)\\
\end{align*}
for some $g(x), h(x) \in \mathbb C(x)\setminus \{0\}$ invariant by the involution $x\mapsto -x$; here $G'\cong \mathbb Z_2^2$ and the genus of the hyperelliptic
curve fixed by any non--trivial involution of $G'$ is odd;\\
\item [{\bf (C.22,22)}] $r=4$, birational maps of $\PP^1\times \PP^1$ to itself preserving the pencil given by the first projection $p_1:  \PP^1\times \PP^1\longrightarrow \PP^1$, generated by 
\begin{align*}
&(x,y)\dasharrow \Big (x, \frac {g(x)}y\Big ),\\
&(x,y)\dasharrow \Big (x, \frac {h(x)y - g(x)}{y - h(x)}\Big)\\
&(x,y)\mapsto (\pm x^{\pm 1},y)\\
\end{align*}
for some $g(x), h(x) \in \mathbb C(x)\setminus \{0\}$ invariant by  $x\mapsto \pm x^{\pm 1}$; here $G'\cong \mathbb Z_2^2$ and the genus of the hyperelliptic
curve fixed by any non--trivial involution of $G'$ is equal to 3 modulo 4.\\
\end{inparaenum}
\end{thm}

\subsection{Blanc's classification in Del Pezzo case}

Here this is the result  (see again \cite [Chapt. 10, Thm. B] {Bl}):

\begin{thm}\label{thm:class2} A subgroup $G\cong \mathbb Z_2^r$, with $r>1$, of the Cremona group of the plane that presents the Del Pezzo case is conjugate, in the Cremona group, to one (and only one) of the following:\\
\begin{inparaenum}
\item[{\bf (4.222) -- (4.2222)}] two groups of automorphisms of the Del Pezzo surface of degree 4 in $\PP^4$, obtained by blowing up the plane  at the four coordinate points plus at the point $(a:b:c)$, that is defined in $\PP^4$ by the two quadratic equations
\begin{align}\label{eq:dp4}
&cx_1^2-ax^2_3 -(a - c)x^2_4 -ac(a - c)x^2_5 = 0\\
&cx_2^2-bx^2_3 -(c - b)x^2_4 -bc(c - b)x^2_5 = 0.\nonumber
\end{align}
The two groups have $r=3$ (case (4.222)), and
\begin{equation}\label{eq:GG}
G=\{(x_1 : x_2 : x_3 : x_4 : x_5) \mapsto (\pm x_1 : \pm x_2 : \pm x_3 : x_4 : x_5)\},
\end{equation}
and $r=4$ (case (4.2222)), and
\begin{equation}\label{eq:GGG}
G=\{(x_1 : x_2 : x_3 : x_4 : x_5) \mapsto (\pm x_1 : \pm x_2 : \pm x_3 : \pm x_4 : x_5)\};
\end{equation}
\item [{\bf (2.G2) -- (2.G22)}] two groups of automorphisms of Del Pezzo surfaces of degree 2 obtained as a double cover of the plane branched along a smooth quartic curve $\Gamma$ of equation $f(x,y,z)=0$. The two groups contain the Bertini involution $\sigma$ and are of the form $G=\langle \sigma \rangle \times H$, where $H\subset 
 {\rm PGL}(3, \mathbb C)$ is a group of automorphisms of
the quartic curve $\Gamma$. 

We have the case $r=2$ (case (2.G2)) in which
$\Gamma$ has equation of the form 
$$
f_4(x, y) + f_2(x, y)z^2 + z^4=0
$$
with $f_i(x,y)$ homogeneous polynomials of degree $i\in \{2,4\}$ and 
$H\cong \mathbb Z_2$ is generated by the involution $(x:y:z)\mapsto (x:y:-z)$. 

We have the case $r=3$ (case (2.G22)) in which
$\Gamma$ has equation of the form 
$$
f_2(x^2, y^2,z^2)=0
$$
with $f_2(x,y,z)$ homogeneous polynomial of degree 2 and 
$H\cong \mathbb Z_2^2$ is generated by the involutions $(x:y:z)\mapsto (x:-y:z)$ and $(x:y:z)\mapsto (x:y:-z)$;\\
\item [{\bf (1.B2.1)}] $r=2$ a group of automorphisms of a Del Pezzo surface of degree 1, that is the double cover of a quadric cone $Q$ in $\PP^3$ branched along a smooth curve $\Gamma$ cut out on $Q$ by a cubic hypersurface and at the vertex of the cone. The group $G$ contains the Geiser involution $\tau$ and $G=\langle \tau\rangle \times H$, with $H\cong \mathbb Z_2$ generated by a projective involution of the cone $Q$ fixing the curve $\Gamma$. 
\end{inparaenum}
\end{thm}

\section{$\mathbb Z_2^r$ covers of the plane} \label{sec:pard}

In this section, following Pardini's paper \cite {Par}, I will recall the theory of $\mathbb Z_2^r$ covers of the plane, and, in view of Blanc's cassification recalled above, I will especially focus on the cases $2\leq r\leq 4$. 

\subsection{Generalities} I will consider the situation of a finite $\mathbb Z_2^r$ cover $f: S\longrightarrow S'$, with $S, S'$ projective irreducible surfaces and $S'$  smooth  (in most cases $S'$ for us will be $\PP^2$). If $G\cong \mathbb Z_2^r$, I will denote by $G^*\cong \mathbb Z_2^r$ the group of characters of $G$. Then
one has the following splitting of $f_*(\mathcal O_S)$ as a sum of line bundles
\begin{equation}\label{eq:reg}
f_*(\mathcal O_S)= \oplus _{\chi\in G^*}L^{-1}_\chi,
\end{equation}
where $L_0\cong \mathcal O_{S'}$. Actually $G$ acts on the sheaf of algebras $f_*(\mathcal O_S)$ and  $G$ acts on $L^{-1}_\chi$ via the character $\chi$. The algebra structure on $f_*(\mathcal O_S)$ is compatible with the action of $G$, so the multiplication
is determined by linear maps 

$$
\mu_{\chi, \chi'}: L^{-1}_\chi\otimes L^{-1}_{\chi'}\longrightarrow L^{-1}_{\chi\chi'}\quad \text{for}\quad \chi, \chi'\in G^*.
$$

Let us denote by $R$ the ramification curve on $S$ and by $D$ the branch curve of $f$ on $S'$. Let  $T$ be an irreducible  component of $R$. Then the \emph{inertia group} $H_T$ of $T$ is defined as
$$
 H_T=\{g\in  G: g x = x, \,\, \forall x \in T\}.
$$

It is known that $H_T$ is a cyclic subgroup of $G$ (see \cite[Lem. 1.1]{Par}), in particular there is a unique non--zero $g_T\in G$ such that $H_T=\langle g_T\rangle$. If $D'$ is any irreducible component of $D$, then all the irreducible components of $f^{-1}(D')$ have the same inertia group, with generator $g_{D'}$. In particular if $D'$ is also reduced, for a general point $p\in D'$, $f^{-1}(p)$ consists of $t:=2^{r-1}$ distinct points $p_1,\ldots , p_t$ where there is a simple ramification, and specifically $g_{D'}(p_i)=p_i$ for  $1\leq i\leq t$. Therefore $D'$ appears with multiplicity $t$ in the branch locus scheme of $f$. 

So we may associate a non--zero element 
 $g_{D'}\in G$ to every irreducible and reduced component $D'$ of $D$. Conversely, for any non--zero $g\in G$, we can denote by $D_g$ the union of all the components of the branch curve whose associated cyclic non--zero element is  $g\in G$. I will call $D_g$ the \emph{$g$--component} of the branch curve.  Notice that given $g$, $D_g$ can be zero. In conclusion, we can write the branch curve as
$$
D=\sum_{g\in G\setminus \{0\}}D_g.
$$

The $\mathbb Z_2^r$ cover $f: S\longrightarrow S'$ is said to be \emph{totally ramified} if  $\mathbb Z_2^r$ is generated by all elements $g$ such that $D_g$ is non--zero. If $S'$ is irreducible, then $S$ is also irreducible if and only if $f: S\longrightarrow S'$ is totally ramified. In this paper I will consider only totally ramified $\mathbb Z_2^r$ covers.

\subsection{Building data} In \cite [Def. 2.1]{Par}, Pardini defines the set of line bundles $\{L_\chi\}_{\chi\in G^*}$ and the set of divisors $\{D_g\}_{g\in G\setminus \{0\}}$ to be the \emph{building data} of the $\mathbb Z_2^r$--cover $f: S\longrightarrow S'$.  I will call the set of divisors $\{D_g\}_{g\in G\setminus \{0\}}$ the \emph{branch data} of the $\mathbb Z_2^r$--cover.

For every character $\chi\in G^*$ and for every non--zero $g\in G$, we may have either $\chi(g)=0$ or $\chi(g)\neq 0$. I set $\epsilon_\chi(g)=0$ in the former case and $\epsilon_\chi(g)=1$ in the latter. Given two characters $\chi, \chi'\in G^*$ and any non--zero $g\in G$, I set 
$$
\epsilon_{\chi,\chi'}(g)= \left\{
\begin{array}{r}
1, \,\, \text{if $\epsilon_\chi(g)+\epsilon_{\chi'}(g)=2$}
\\
0, \,\, \text{if $\epsilon_\chi(g)+\epsilon_{\chi'}(g)<2$}
\end{array}
\right.
$$
Pardini proves that
\begin{equation}\label{eq:prod}
L_\chi\otimes L_{\chi'}\sim L_{\chi\chi'}+\sum_{g\in G\setminus \{0\}}\epsilon_{\chi,\chi'}(g) D_g, \quad \text{for}\quad \chi, \chi'\in G^*.
\end{equation}
(where $\sim$ denotes linear equivalence). 

One of the main results in \cite {Par} is Theorem 2.1 of that paper, to the effect that a set of building data satisfying \eqref 
{eq:prod} uniquely determines a $\mathbb Z_2^r$ cover $f: S\longrightarrow S'$ up to isomorphisms. 

In fact all building data are redundant for recovering the cover. Indeed Pardini defines in \cite [Def. 2.3]{Par} a \emph{reduced set of building data} as follows. We know that $G^*\cong \mathbb Z_2^r$, so we can determine $\chi_1,\ldots, \chi_r$ a minimal set of generators of $G^*$ and one writes $L_i=L_{\chi_i}$, for $1\leq i\leq r$. Then set of line bundles $\{L_i\}_{1\leq i\leq r}$ and the set of divisors $\{D_g\}_{g\in G\setminus \{0\}}$ is a reduced set  of building data of the $\mathbb Z_2^r$--cover $f: S\longrightarrow S'$.  In view of \eqref {eq:prod} a reduced set of building data suffices to recover a full set of building data. More precisely, one proves that 
\begin{equation}\label{eq:prodred}
2L_i \sim \sum_{g\in G\setminus \{0\}} \epsilon_{\chi_i}(g)D_g,\,\, 1\leq i\leq r
\end{equation}
and a reduced set  of building data verifying \eqref {eq:prodred} uniquely determines a $\mathbb Z_2^r$ cover $f: S\longrightarrow S'$ up to isomorphisms (see \cite [Prop. 2.1]{Par}).

 Note that, if $S'$ has no 2--torsion, then the branch data determine, via \eqref {eq:prodred}, the line bundles $L_i$, for $1\leq i\leq r$, and hence all the line bundles $L_\chi$, for $\chi\in G^*$, via \eqref {eq:prod}. 
 
 \subsection{Building data for $2\leq r\leq 4$}
 
 In this section I collect the explicit form of the building data and their relations  \eqref {eq:prod} for $\mathbb Z_2^r$ covers, for $2\leq r\leq 4$, that, in view of Blanc's classification, I will need later. 
 
 I will consider, for all $g\in \mathbb Z_2^r$, the obvious notation
$g=(\gamma_1,\ldots, \gamma_r)$, with $\gamma_i\in \{0,1\}$. If I use $g$ as an index like $a_g$, I simply write $a_{\gamma_1\ldots \gamma_r}$.
 
 \subsubsection{The building data for $r=2$}\label{ssec:2}  The  building data of a $\mathbb Z_2^2$--cover $f: S\longrightarrow \PP^2$ are:\\
\begin{inparaenum}[(i)]
\item three plane curves $D_{10}, D_{01}, D_{11}$ of degrees 
$d_{10}, d_{01}, d_{11}$ respectively;\\
\item three line bundles  $L_{10}\cong \mathcal O_{\PP^2}(a_{10})$, 
$L_{01}\cong \mathcal O_{\PP^2}(a_{01})$ and $L_{11}\cong \mathcal O_{\PP^2}(a_{11})$;\\
\item the relations \eqref {eq:prod} read
$$
\begin{array}{c}
2L_{10} \sim D_{10}+D_{11}
\\
2L_{01} \sim D_{01}+D_{11}
\\
2L_{11} \sim D_{01}+D_{10}
\\
L_{10}\otimes L_{01}\sim L_{11}+D_{11}
\\
L_{10}\otimes L_{11}\sim L_{01}+D_{01}
\\
L_{01}\otimes L_{11}\sim L_{10}+D_{10}
\end{array}
$$
that imply
$$
\begin{array}{r}
2a_{10} = d_{10}+d_{11}
\\
2a_{01} = d_{01}+d_{11}
\\
2a_{11} = d_{10}+d_{01}
\end{array}
$$
so that $d_{10}, d_{01},d_{11}$ have the same parity.\end{inparaenum}

\subsubsection{The building data for $r=3$} \label{ssec:bd3} The  building data of a $\mathbb Z_2^3$--cover $f: S\longrightarrow \PP^2$ are:\\
\begin{inparaenum}[(i)]
\item seven plane curves $D_g$ of degree $d_g$, for all $g\in  \mathbb Z_2^3\setminus \{0\}$;\\
\item seven line bundles  $L_g$, for all $g\in  \mathbb Z_2^3\setminus \{0\}$;\\
\item the relations \eqref {eq:prod} read
$$
\begin{array}{c}
2L_{100} \sim D_{100}+D_{110}+D_{101}+D_{111}
\\
2L_{010} \sim D_{010}+D_{110}+D_{011}+D_{111}
\\
2L_{001} \sim D_{001}+D_{101}+D_{011}+D_{111}
\\
L_{110}\sim L_{100}\otimes L_{010}-D_{110}-D_{111}
\\
L_{101}\sim L_{100}\otimes L_{001}-D_{101}-D_{111}
\\
L_{011}\sim L_{010}\otimes L_{001}-D_{011}-D_{111}
\\
L_{111}\sim L_{110}\otimes L_{001}-D_{011}-D_{101}.
\end{array}
$$
\end{inparaenum}

\subsubsection{The building data for $r=4$} \label{ssec:bd3} The  building data of a $\mathbb Z_2^4$--cover $f: S\longrightarrow \PP^2$ are:\\
\begin{inparaenum}[(i)]
\item 15 plane curves $D_g$ of degree $d_g$, for all $g\in  \mathbb Z_2^3\setminus \{0\}$;\\
\item 15 line bundles  $L_g$, for all $g\in  \mathbb Z_2^3\setminus \{0\}$;\\
\item the relations \eqref {eq:prod} read
\begin{equation}\label{eq:data}
\begin{array}{c}
2L_{1000} \sim D_{1000}+D_{1100}+D_{1010}+D_{1001}+D_{1110}+D_{1101}+D_{1011}+D_{1111}
\\
2L_{0100} \sim D_{0100}+D_{1100}+D_{0110}+D_{0101}+D_{1110}+D_{1101}+D_{0111}+D_{1111}
\\
2L_{0010} \sim D_{0010}+D_{1010}+D_{0110}+D_{0011}+D_{1110}+D_{1011}+D_{0111}+D_{1111}
\\
2L_{0001} \sim D_{0001}+D_{1001}+D_{0101}+D_{0011}+D_{1101}+D_{1011}+D_{0111}+D_{1111}
\\
L_{1100}\sim L_{1000}\otimes L_{0100}-D_{1100}-D_{1110}-D_{1101}-D_{1111}
\\
L_{1010}\sim L_{1000}\otimes L_{0010}-D_{1010}-D_{1110}-D_{1011}-D_{1111}
\\
L_{1001}\sim L_{1000}\otimes L_{0001}-D_{1001}-D_{1101}-D_{1010}-D_{1111}
\\
L_{0110}\sim L_{0100}\otimes L_{0010}-D_{0110}-D_{1110}-D_{0111}-D_{1111}
\\
L_{0101}\sim L_{0100}\otimes L_{0001}-D_{0101}-D_{1101}-D_{0111}-D_{1111}
\\
L_{0011}\sim L_{0010}\otimes L_{0001}-D_{0011}-D_{1011}-D_{0111}-D_{1111}
\\
L_{1110}\sim L_{1000}\otimes L_{0110}-D_{1010}-D_{1100}-D_{1101}-D_{1010}
\\
L_{1011}\sim L_{1000}\otimes L_{0011}-D_{1010}-D_{1001}-D_{1110}-D_{1101}
\\
L_{0111}\sim L_{0110}\otimes L_{0001}-D_{0101}-D_{0011}-D_{1011}-D_{1101}
\\
L_{1101}\sim L_{1100}\otimes L_{0001}-D_{1001}-D_{0101}-D_{0111}-D_{1011}
\\
L_{1111}\sim L_{1100}\otimes L_{0011}-D_{1010}-D_{1001}-D_{0110}-D_{0101}.
\end{array}
\end{equation}
\end{inparaenum}

\subsection{Normalization} \label {sss:norm} Given a $\mathbb Z_2^r$--cover $f: S\longrightarrow S'$ with  building data $\{L_\chi\}_{\chi\in G^*}$ and  branch data $\{D_g\}_{g\in G\setminus \{0\}}$, $S$ is normal if and only if each irreducible component of one of the non--zero divisors $\{D_g\}_{g\in G\setminus \{0\}}$, appears in the  branch curve $D=\sum_{g\in G\setminus \{0\}}D_g$ with multiplicity 1 (see \cite [Cor. 3.1]{Par}). Pardini describes in \cite [p. 203]{Par} the  process that brings to the normalization of a non--normal $S$ as above. In our situation, this process reduces to the following two steps.\\
\begin{inparaenum}
\item [{\bf Step 1:}] Suppose there is an irreducible curve $\Delta$ that appears in the  branch curve $D$ with multiplicity greater than 1. Consider first a non--zero element $g\in G$ such that one has $D_g=k_g\Delta+\bar D_g$, where $k_g\geq 2$ and $\bar D_g$ not containing $\Delta$ as a component. Write $k_g=2q_g+r_g$, with $0\leq r_g\leq 1$. Then we define
$D'_g=D_g-2q_g\Delta$ and for all $\chi\in G^*\setminus \{0\}$ define
$$
L_\chi'= \left\{
\begin{array}{r}
L_\chi, \,\, \text{if $\chi(g)=0$}
\\
L_\chi(-q_g\Delta), \,\, \text{if $\chi(g)\neq 0$.}
\end{array}
\right.
$$
Repeat this for every non--zero element $g\in G$ such that one has $D_g=k_g\Delta+\bar D_g$, where $k_g\geq 2$ and for the other non--zero $g\in G$ we set 
$D'_g=D_g$ and do not change the $L_\chi$s. Then repeat this for every irreducible curve $\Delta$ that appears in the  branch curve $D$ with multiplicity greater than 1. In this way one constructs a new set of data that are seen to be building data of a new $\mathbb Z_2^r$--cover of $S'$. For these new data, that by abuse of notation I still denote by $\{L_\chi\}_{\chi\in G^*\setminus \{0\}}$ and  $\{D_g\}_{g\in G\setminus \{0\}}$, one has that for any non--zero $g\in G$, $D_g$ is  reduced. \\
\item [{\bf Step 2:}] Suppose next that there is a divisor $\Delta$ such that there are two distinct non--zero $g,h\in G$ such that $\Delta$ is contained in $D_g$ and $D_h$. 
Then we set $k=g+h$ and  $D'_g=D_g-\Delta$, 
$D'_h=D_h-\Delta$, $D'_k=D_k+\Delta$, and $D'_j=D_j$ for all $j\in G$ different from $g,h,k$. Moreover we set $L'_\chi=L_\chi(-\Delta)$, if $\chi(g)=\chi(h)=1$, whereas
we set $L'_\chi=L_\chi$ otherwise. Then repeat this process any time there is a non--zero irreducible curve  appearing in two curves $D_g, D_h$. 
One sees that this process ends (see \cite [Prop. 3.2]{Par}) and that one has at the end a new set of building data that correspond to a  $\mathbb Z_2^r$--cover $\bar f: \bar S\longrightarrow S'$, with $\bar S$ normal. 
\end{inparaenum}

Applying this process any time we have a $\mathbb Z_2^r$--cover $f: S\longrightarrow S'$ with $S'$ smooth, we may (and will) assume that $S$ is normal.  

\subsection{Normal singularities} Suppose next we have a $\mathbb Z_2^r$--cover $f: S\longrightarrow S'$ with $S'$ smooth and $S$ normal.  The following result is \cite [Prop. 3.1]{Par}:

\begin{proposition}\label{prop:par} Suppose  we have a $\mathbb Z_2^r$--cover $f: S\longrightarrow S'$ with $S'$ smooth and $S$ normal, with branch curve
$D=\sum_{g\in G\setminus\{0\}} D_g$. Then $S$ is singular over a point $x\in S$ if and only if $x\in B$ and moreover $x\in D_{g_i}$, with $1\leq i\leq r$, with $g_i\in G$ distinct, and either the obvious map 
$$
\oplus_{i=1}^r \langle g_i\rangle \longrightarrow G
$$
is not an injection or, if $b_i=0$ is a local equation of $D_{g_i}$ around $x$, for $1\leq i\leq r$, then $db_1,\ldots, db_r$ are not linearly independent in  the dual of the Zariski tangent space of $S$ and $x$.
\end{proposition}

In particular,  $S$ is smooth if the divisors $D_g$ are smooth for all non--zero $g\in G$ and $D$ has normal crossings. 

If $S$ is normal and singular, one can resolve the singularities of $S$ in the following way. Let $x\in S'$ be a point over which $S$ is singular. Let $\pi: \tilde S'\longrightarrow S' $ be the blow--up of $S'$ at $x$. Consider the Cartesian diagram
$$
 \label{eq:compos0}
  \xymatrix{
    \tilde S \ar[d]_{f'}  \ar[rr]^{\pi'}  &  &  \ar[d]^{f} S  \\
    \tilde S'  \ar[rr]_{\pi } & &S'
}
$$
so that $f': \tilde S\longrightarrow \tilde S'$ is a $\mathbb Z_2^r$--cover whose building data are just the pull back to $\tilde S'$ of the building data of $f: S\longrightarrow S'$ via $\pi$. Then normalize $\tilde S$ with the recipe given above, and proceed further. 

\subsection{Invariants} One has (see \cite [Prop. 4.2]{Par}):

\begin{proposition}\label{prop:inv} Consider a $G$--cover $f: S\longrightarrow S'$, with $G\cong \mathbb Z_2^r$ and $S, S'$ both smooth. Then one has
\begin{equation}\label{eq:chi}
\chi(\mathcal O_S)=2^r\chi(\mathcal O_{S'})+ \frac 12 \sum_{\chi\in G^*\setminus \{0\}} L_\chi \cdot (L_\chi+K_{S'}),
\end{equation}
and
\begin{equation}\label{eq:K2}
K_S^2=2^r\Big ( K_{S'}+\sum_{g\in G\setminus \{0\}} \frac 12 D_g\Big )^2.
\end{equation}

\end{proposition}

\section{The classification in the invariant  conic bundle case}\label{sec:classinv}

In this section I consider a $\mathbb Z_2^r$--cover $f: S\longrightarrow \PP^2$ with $S$ rational and normal (and $r>1$), and moreover I will consider the invariant conic bundle case. As we saw in Theorem \ref {thm:class1}, we have $2\leq r\leq 4$. I will consider the minimal desingularization $\phi: S'\longrightarrow S$. I will assume in this section that the pair $(S', \mathbb Z_2^r)$ is a minimal $\mathbb Z_2^r$--surface and that ${\rm Pic}(S')^{\mathbb Z_2^r}$ has rank 2, generated over $\mathbb Q$ by the class of $K_{S'}$ and by the class of a  general fibre $F$ of a pencil $|F|$ of rational curves that is fixed by the  $\mathbb Z_2^r$ action (see Proposition \ref {prop:bla}). 

The main remark is that the image of $|F|$ to the plane via the composite map $f\circ \phi$ is a pencil of rational curves. Then, up to a Cremona transformation of the plane, this pencil can be assumed to be the pencil $\mathcal F$ of lines passing through a given point $p\in \PP^2$. I will assume from now on that  this is the case.

\subsection{The case $r=2$} Here I have $G=\mathbb Z_2^2$.

\subsubsection{The three intermediate double covers}\label{ssec:inter} Given a $\mathbb Z_2^2$--cover $f: S\longrightarrow \PP^2$, we have the three non--zero elements $(1,0), (0,1), (1,1)$  of $\mathbb Z_2^2$. By modding out $S$ by these elements, we have three double covers 
$$
f_{10}: S\longrightarrow S_{10}, \,\, f_{01}: S\longrightarrow S_{01}, \,\, f_{11}: S\longrightarrow S_{11}.
$$
 The map $f: S\longrightarrow \PP^2$ factors through each one of these double covers, i.e., we have three double covers
\begin{equation}\label{eq:rob}
g_{10}: S_{10}\longrightarrow \PP^2,\,\, g_{01}: S_{01}\longrightarrow \PP^2, \,\, g_{11}: S_{11}\longrightarrow \PP^2
\end{equation}
so that $f=g_{10}\circ f_{10}=g_{01}\circ f_{01}=g_{11}\circ f_{11}$. The branch curve of the double cover $g_{10}: S_{10}\longrightarrow \PP^2$ is the curve $D_{01}+D_{11}$ and similarly for the double covers $g_{01}: S_{01}\longrightarrow \PP^2$ and $g_{11}: S_{11}\longrightarrow \PP^2$. From this we see that at most one of the divisors $D_{10}, D_{01}, D_{11}$ is zero, otherwise $S$ would not be irreducible. Equivalently, if at least two among $D_{10}, D_{01}, D_{11}$ are zero, then the cover is not totally ramified. 

\subsubsection{The behaviour of $G'$ for $r=2$} Remember the sequence \eqref {eq:split}. Looking at Theorem \ref {thm:class1}, we have five cases to take care of:\\
\begin{inparaenum}[(i)]
\item cases in which $G'=0$, corresponding to {\bf (P1.221)} and {\bf (P1.22.1)} in Theorem \ref {thm:class1};\\
\item cases in which $G'\cong \mathbb Z_2$, corresponding to {\bf (0.22)} and {\bf (C.2,21)} in Theorem \ref {thm:class1};\\
\item case in which $G'=G=\mathbb Z_2^2$, corresponding to {\bf (C.22)} in Theorem \ref {thm:class1}.
\end{inparaenum}

 I assume that the invariant pencil $|F|$ by the $\mathbb Z_2^2$ action is mapped to the pencil $\mathcal F$ of lines through the point $p\in \PP^2$. Let $\ell\in \mathcal F$ be a general line in the pencil. Then, according to the above three cases for $G'$, we have the following possibilities:\\
\begin{inparaenum}[(i)]
\item the pull back of $\ell$ via $f$ consists of four distinct  curves in $|F|$, in which case $G'=0$, and there is no ramification point on $\ell$ (see cases {\bf (P1.221)} and {\bf (P1.22.1)});\\
\item the pull back of $\ell$ via $f$ consists of two distinct  curves in $|F|$, in which case $G'\cong \mathbb Z_2$, and (by Riemann--Hurwitz formula) there are 2 ramification points on $\ell$, each has to be counted with multiplicity 2 (see cases {\bf (0.22)} and {\bf (C.2,21)});\\
\item the pull back of $\ell$ via $f$ consists of a unique curve in $|F|$, in which case $G'=G\cong  \mathbb Z_2^2$, and (by Riemann--Hurwitz formula) there are 3 ramification points on $\ell$, each has to be counted with multiplicity 2 (see case {\bf (C.22)}).
\end{inparaenum}

\subsubsection{The classification in the $G'=0$ case for $r=2$}

\begin{example}\label{ex:class}  Consider the Hirzebruch surface $\mathbb F_1$ (the blow--up of $\PP^2$ at a point $p$) with its structure map $\mu: \mathbb F_1\longrightarrow \PP^1$ given by the pull back on $\mathbb F_1$ of the pencil $\mathcal F$. There is a totally ramified $\mathbb Z_2^2$--cover $g: \PP^1\longrightarrow \PP^1$ with branch data given by three distinct points of $\PP^1$. There is then a cartesian diagram
$$
 \label{eq:compos}
  \xymatrix{
    Y \ar[d]_{f'}  \ar[rr]^{\nu}  &  &  \ar[d]^{g} \PP^1  \\
    \mathbb F_1  \ar[rr]_{\mu} & & \PP^1
}
$$
with $f': Y\longrightarrow \mathbb F_1$ a smooth $\mathbb Z_2^2$--cover and 
$\nu: Y\longrightarrow \PP^1$  a conic bundle structure, so that $Y$ is rational.  The 
$\mathbb Z_2^2$--cover $f': Y\longrightarrow \mathbb F_1$ is branched along three distinct fibres of $\nu$.

If $E$ is the $(-1)$--curve on $\mathbb F_1$, the pull back of $E$ on $Y$ is an irreducible rational curve $\Delta$, with $\Delta^2=-4$. By contracting this curve to a point we have a birational map $\pi': Y\longrightarrow S$ (with $S$ having a unique singular quadruple point of type $\frac 14(1,1)$) with a commutative diagram
\begin{equation}
 \label{eq:compos}
  \xymatrix{
    Y \ar[d]_{f'}  \ar[rr]^{\pi'}  &  &  \ar[d]^{f} S  \\
    \mathbb F_1  \ar[rr]_{\pi } & & \PP^2
}
\end{equation}
where $f: S\longrightarrow \PP^2$ is a $\mathbb Z^2_2$--cover with branch curve given by three distinct lines passing through the point $p$.  In conclusion $Y\cong \mathbb F_4$ (this is confirmed by the fact that $K_Y^2=8$ as it follows from an application of \eqref {eq:K2}) and $S$ is isomorphic to a cone over a rational normal curve of degree 4. \end{example}

\begin{proposition}\label{prop:G'=0} Consider a $G$--cover $f: S\longrightarrow \PP^2$, with $G\cong \mathbb Z_2^2$,  $S$ normal and $G'=0$, presenting the invariant conic bundle structure. Then, up to  a Cremona transformation, $f: S\longrightarrow \PP^2$ coincides with the cover considered in Example \ref {ex:class}, i.e.,
the image of the invariant pencil is the pencil $\mathcal F$ of lines through a point $p$ of the plane and the branch curve consists of three distinct lines of $\mathcal F$. More precisely, with the  notation as in Section \ref {ssec:2}, $D_{10}, D_{01}, D_{11}$ are three distinct lines of $\mathcal F$. 

Conversely, a $G$--cover $f: S\longrightarrow \PP^2$, with $G\cong \mathbb Z_2^2$ with branch curve consisting of three distinct lines in a pencil, then $S$ is rational and it is as in Example \ref {ex:class}.
\end{proposition}

\begin{proof} As we saw, in this case the general line in $\mathcal F$ contains no branch point off $p$, and this implies that the curves $D_{10}, D_{01}, D_{11}$, of degrees $d_{10}, d_{01}, d_{11}$ respectively, are either 0 (at most one of them can be 0) or have a point of multiplicity 
$d_{10}, d_{01}, d_{11}$ respectively in $p$, so they consist of distinct lines through the point $p$. 

Consider the three intermediate double covers in \eqref {eq:rob}. Each of the surfaces $S_{01}, S_{10}, S_{11}$ is rational and it is a double cover of $\PP^2$ branched along a set of lines passing through the point $p$. Precisely  $g_{10}: S_{10}\longrightarrow \PP^2$ is branched along the $d_{01}+d_{11}$ lines of $D_{01}+D_{11}$, and similarly for $g_{01}: S_{01}\longrightarrow \PP^2$ and $g_{11}: S_{11}\longrightarrow \PP^2$. By the rationality of $S_{01}, S_{10}, S_{11}$, we deduce that $d_{10}=d_{01}=d_{11}=1$, as wanted. 

The last part of the statement follows by the argument in Example \ref {ex:class}.\end{proof}

\begin{remark}\label{rem:quadr} Notice that, in both cases {\bf (P1.221)} and {\bf (P1.22.1)} in Theorem \ref {thm:class1}, $\mathbb Z_2^2$ acts on the smooth quadric $Q=\PP^1\times \PP^1$ and determines $\mathbb Z_2^2$--covers $f: Q\longrightarrow Q$ that are ramified on three lines of the same ruling. By projecting down $Q$ to a plane from a general point of $Q$, one gets the cover of Example \ref {ex:class}. 
\end{remark}

\subsubsection{The classification in the $G'=\mathbb Z_2$ case for $r=2$}

\begin{proposition}\label{prop:G'2} Consider a $G$--cover $f: S\longrightarrow \PP^2$, with $G\cong \mathbb Z_2^2$, with $S$ normal and $G'=\mathbb Z_2$, presenting the invariant conic bundle case. Then, up to  a Cremona transformation, the image of the invariant pencil is the pencil $\mathcal F$ of lines through a point $p$ of the plane and the branch curve consists of two distinct lines of $\mathcal F$ and of a curve of odd degree $d$ with a point of multiplicity $d-2$ at $p$. More precisely, with the  notation as in Section \ref {ssec:2}, $D_{10}, D_{01}$ are two distinct lines of $\mathcal F$ and $D_{11}$ is either a line not passing through $p$ or a curve of odd degree $d\geq 2$, neither containing $D_{10}$ nor  $D_{01}$,  with a point of multiplicity $d-2$ at $p$.  

Conversely, any $\mathbb Z_2^2$--cover $f: S\longrightarrow \PP^2$, with the above branch data and $G'=\mathbb Z_2$  is such that $S$ is rational and normal, presenting the invariant conic bundle structure case. 
\end{proposition}

\begin{proof} We can assume that the non--zero element of $G'$ is $(1,1)$. Consider  the double cover $f_{11}: S\longrightarrow S_{11}$. The image via $f_{11}$ of the invariant pencil, is a rational pencil $\mathcal F'$ of rational curves.
The pull--back via $f_{11}$ of a general curve in $\mathcal F'$ is a unique curve of the invariant pencil. 

Consider next the double cover $g_{11}: S_{11}\longrightarrow \PP^2$, that maps $\mathcal F'$ to $\mathcal F$. This induces a double cover $\mathcal F'\longrightarrow \mathcal F$, such that the pull back of a general line $\ell$ in $\mathcal F$ is the union of two rational curves that map isomorphically to $\ell$. This implies that $g_{11}: S_{11}\longrightarrow \PP^2$ is branched along two distinct lines passing through $p$, which implies that $D_{10}, D_{01}$ are two distinct lines of $\mathcal F$. Then $D_{11}$ is a curve of odd degree $d$ with a point of multiplicity $c$ in $p$. Since, as we saw, on the general line $\ell$ in $\mathcal F$ there are two branch points, one has $d-2\leq c\leq d-1$. Precisely, if $c=d-2$ then $p$ is not a branch point on the general line of $\mathcal F$, whereas if $c=d-1$ then $p$ is  a branch point on the general line of $\mathcal F$.

Let us prove that if $c=d-1$, then up to a Cremona transformation of the plane, we can suppose that $d=1$ and  $D_{11}$ does not contain $p$. I will imitate the proof in \cite [Prop. 8.2.3]{Cal} and therefore I will be brief. Indeed, assume $c=d-1$. Then I can write
$$
D_{11}= C+ R_1+\cdots + R_{d-e}
$$
with $C$ irreducible of degree $e$ with an ordinary point of multiplicity $e-1$ at $p$, and $R_1,\ldots, R_{d-e}$ distinct lines passing through $p$. 

First I prove by induction that, up to Cremona transformations  we may assume that $e=1$ (i.e., $C$ is a line not passing through $p$). In fact, if $d>1$, choose a general point $q\in C$ and an infinitely near point $r$ to $p$ a
on $C$. Then make the quadratic transformation based at $p,q,r$. This drops drops the degree of $C$ by 1 and adds one more line through $p$ to $R_1,\ldots, R_{d-e}$. Repeating this process, we can reduce to the case in which
$$
D_{11}= R+ R_1+\cdots + R_{d-1}
$$
with $R$ a line not passing through $p$ and $R_1,\ldots, R_{d-1}$ lines through $p$. 

Next I prove by induction that we can get rid of the lines $R_1,\ldots, R_{d-1}$. In fact, if $d>1$, choose the intersection point $q$ of $R$ with $R_{d-1}$ and the point $r$ infinitely near to $p$ along $R_{d-2}$. Then make the quadratic transformation based at $p,q,r$. This eliminates the lines $R_{d-2}, R_{d-1}$. Since $d$ is odd, after $\frac {d-1}2$ steps we get rid of the lines $R_1,\ldots, R_{d-1}$, as wanted. 

To finish our proof I have to prove the last part of the statement, i.e.,  that any $\mathbb Z_2^2$--cover $f: S\longrightarrow \PP^2$, with the above branch data and $G'=\mathbb Z_2$ is such that $S$ is rational. Certainly $S$ has a pencil $\mathcal G$ of rational curves, the pull--back via $f$ of the pencil of lines in $\mathcal F$. To finish the proof, I have to show that $\mathcal G$ is rational. Consider again  the double cover $f_{11}: S\longrightarrow S_{11}$. The image via $f_{11}$ of the invariant pencil, is a pencil $\mathcal F'$ of rational curves. The double cover $g_{11}: S_{11}\longrightarrow \PP^2$, that has branch curve $D_{10}+D_{01}$,
maps two--to--one the pencil $\mathcal F'$ to $\mathcal F$, and the ramification occurs precisely at the two points corresponding to $D_{10}$ and $D_{01}$. This implies that $\mathcal F'$ is rational. Moreover, since $G'=\mathbb Z_2$, the pull--back via $f_{11}$ of a general curve in $\mathcal F'$ is a unique curve of the invariant pencil. This proves that $\mathcal G$ is rational as well as $\mathcal F'$, and this ends our proof. 
\end{proof}

\begin{remark}\label{rem:ot} (i) Let me keep the notation of the proof of  Proposition \ref {prop:G'2}. I show another way of concluding 
the proof, under some generality hypotheses. Assume  in fact we are in a \emph{general situation}, i.e., that $D_{11}$ is irreducible, smooth off $p$, that it has an ordinary singularity at $p$ and that $D_{10}, D_{01}$ are not tangent to $D_{11}$ off $p$. As we saw in the proof of the last part of Proposition \ref {prop:G'2}, we know that $S$ has a pencil $\mathcal G$ of rational curve, so $S$ has Kodaira dimension $-\infty$. To finish the proof, the only thing to check is that $S$ is regular. 

Let us blow--up the plane at $p$, getting $\pi: \mathbb F_1\longrightarrow \PP^2$, with the $(-1)$--curve $E$ of $\mathbb F_1$ contracted at $p$. Consider  the Cartesian diagram \eqref {eq:compos}, with $f': Y\longrightarrow \mathbb F_1$ a $\mathbb Z_2^2$--cover with building data just pull--back via $\pi$ of the building data of $f: S\longrightarrow \PP^2$. Let us denote by $D'_{10}, D'_{01}, D'_{11}$ the strict transforms of $D_{10}, D_{01}, D_{11}$ via $\pi$. Then the branch data of $f': Y\longrightarrow \mathbb F_1$ are 
$$
\bar D_{10}=D'_{10}+E, \,\, \bar D_{01}=D'_{01}+E, \,\, \bar D_{11}=D'_{11}+(d-2)E.
$$
As a consequence $Y$ is not normal and we have to normalize it. Since $d$ is odd, $d-2$ is also odd. Then the normalization process, as described in Section \ref {sss:norm}, imposes to change the branch data into
\begin{equation}\label{eq:bd}
\bar D'_{10}=D'_{10}, \,\, \bar D'_{01}=D'_{01}, \,\, \bar D'_{11}=D'_{11}.
\end{equation}

In the above general situation, after normalizing, we in fact have a desingularization  $\bar S$ of $S$, that is the  $\mathbb Z_2^2$--cover of $\mathbb F_1$ with branch data given by \eqref {eq:bd}. Then I can apply Proposition \ref {prop:inv} to compute $\chi(\mathcal O_{\bar S})$. Let us denote by $|P|$ the ruling of $\mathbb F_1$. Then we have
$$
2L_{10}\sim 2L_{01}\sim(d+1)P+2E, \,\,  2L_{11}\sim 2P
$$
and $K_{\mathbb F_1}=-2E-3P$. So applying \eqref {eq:chi} we easily see that have $\chi(\mathcal O_{\bar S})=1$, as wanted. 

If one is not in the general situation as above, the argument, though a bit more involved, is similar but I do not dive into this here and leave it to the reader. \medskip 

(ii) I conjecture that a result stronger than Proposition \ref {prop:G'2} may hold, namely that, if $d\geq 3$, up to a Cremona transformation of the plane, one can assume that $D_{11}$ is irreducible, that has an ordinary point of multiplicity  $d-2$ at $p$ and at most one node besides the singularity at $p$. The argument for proving this should be similar to the one in \cite [Prop. 8.2.4]{Cal}. However I do not dwell on this here. \medskip

(iii) In the case $d=1, c=0$ of Proposition \ref {prop:G'2} the branch data of the $\mathbb Z_2^2$--cover $f: S\longrightarrow \PP^2$ are given by three distinct lines in the plane not passing through the same point. In this case $S$ is smooth and a simple application of \eqref {eq:K2} shows that $K^2_S=9$, so that $S=\PP^2$. Up to projective transformations, this is the $\mathbb Z_2^2$--cover given by the map
$$
(x_0:x_1:x_2)\in \PP^2\longrightarrow (x^2_0:x^2_1:x^2_2)\in \PP^2
$$
and this is easily seen to correspond to  case {\bf (0.22)} in Theorem \ref {thm:class1}. 

The other cases in Proposition \ref {prop:G'2} correspond to  case {\bf (C.2,21)} in Theorem \ref {thm:class1}. Going back to the proof of Proposition \ref {prop:G'4}, the pull back $\Delta$ to $S_{11}$ via the double cover $g: S_{11}\longrightarrow \PP^2$ of the curve $D_{11}$  (that, as I conjectured in (ii) above, should be irreducible) is the branch curve of the double cover $f_{11}: S\longrightarrow S_{11}$. On $S_{11}$ we have the rational pencil $\mathcal F'$ of rational curves, whose pull back on $S$ is the invariant pencil of rational curves, and precisely the pull--back via $f_{11}$ of a general curve in $\mathcal F'$ is a unique curve of the invariant pencil. This implies that the curves in $\mathcal F'$ intersect $\Delta$ each in two points, so that $\Delta$ is  hyperelliptic. This is the hyperelliptic curve fixed by $G'$ mentioned in case {\bf (C.2,21)} of Theorem \ref {thm:class1}.
\end{remark} 

\subsubsection{The classification in the $G'=G=\mathbb Z_2^2$ case for $r=2$}

\begin{proposition}\label{prop:G'4} Consider a $G$--cover $f: S\longrightarrow \PP^2$, with $G\cong \mathbb Z_2^2$, with $S$ normal and $G'=\mathbb Z_2^2$, presenting the invariant conic bundle  case. Then, up to  a Cremona transformation, the image of the invariant pencil is the pencil $\mathcal F$ of lines through a point $p$ of the plane and, with the above notation, $D_{10}, D_{01}, D_{11}$ are curves of degree $d_{10}, d_{01}, d_{11}$ respectively with the same parity, with a point of multiplicity $d_{10}-1, d_{01}-1, d_{11}-1$ at $p$ respectively, with no common components (and for only one of these curves the point $p$ can even be of multiplicity one more). 

Conversely, any $\mathbb Z_2^2$--cover $f: S\longrightarrow \PP^2$, with the above branch data and $G'=\mathbb Z_2^2$ is such that $S$ is rational and normal, presenting the invariant conic bundle structure case. 
\end{proposition}

\begin{proof} Let us set for simplicity $(d_{10}, d_{01}, d_{11})=(a,b,c)$ and let $(\lambda, \mu, \nu)$ be the multiplicities of $D_{10}, D_{01}, D_{11}$ at $p$. As we know, the general line $\ell$ in $\mathcal F$ has three branch points on it. This means that:\\
\begin{inparaenum}[(a)]
\item either $\ell$ intersects the sum of the curves $D_{10}, D_{01}, D_{11}$ in three points off $p$, and $p$ is not a ramification point;\\
\item or $\ell$ intersects the sum of the curves $D_{10}, D_{01}, D_{11}$ in two points off $p$, and $p$ is  a ramification point.
\end{inparaenum}

Let us examine first case (a). In this case we may have (up to reordering $D_{10}, D_{01}, D_{11}$) the following cases:\\
\begin{inparaenum}
\item [(a1)] $(\lambda, \mu, \nu)=(a-1,b-1,c-1)$;\\
\item [(a2)] $(\lambda, \mu, \nu)=(a,b-1,c-2)$;\\
\item [(a3)] $(\lambda, \mu, \nu)=(a,b,c-3)$.
\end{inparaenum}

I claim that cases (a2) and (a3) cannot happen. Look at case (a2). Let us  blow--up the plane at $p$, getting $\pi: \mathbb F_1\longrightarrow \PP^2$, with the $(-1)$--curve $E$ of $\mathbb F_1$ contracted at $p$. Consider  the Cartesian diagram \eqref {eq:compos}, with $f': Y\longrightarrow \mathbb F_1$ a $\mathbb Z_2^2$--cover with building data just pull--back by $\pi$ of the building data of $f: S\longrightarrow \PP^2$. Let $D'_{10}, D'_{01}, D'_{11}$ be the strict transforms on $\mathbb F_1$ of $D_{10}, D_{01}, D_{11}$. The branch data of $f': Y\longrightarrow \mathbb F_1$ are
$$
\bar D_{10}= D'_{10}+aE , \,\, \bar D_{01}\sim D'_{01}+(b-1)E, \,\, \bar D_{11}=D'_{11}+(c-2)E.
$$
If $a,b,c$ are even, after normalizing $Y$ the branch data become
\begin{equation}\label{eq:loi}
\bar D_{10}= D'_{10} , \,\, \bar D_{01}\sim D'_{01}+E, \,\, \bar D_{11}=D'_{11}
\end{equation}
and this is impossible, because $E$ appears in the branch curve and therefore $p$ is a branch point on a general line in $\mathcal F$, a case I have excluded for the moment. If $a,b,c$ are odd, in the process of normalizing $Y$ the branch data first become
$$
\bar D_{10}= D'_{10}+E , \,\, \bar D_{01}\sim D'_{01}, \,\, \bar D_{11}=D'_{11}+E
$$
and then again become as in \eqref {eq:loi}, and this leads to a contradiction as well. 
The argument for the case (a3) is completely similar and can be left to the reader. 

The case (a1) instead, corresponds exactly to the situation described in the statement of the proposition. 

To complete the argument in this case, let us prove that any $\mathbb Z_2^2$--cover $f: S\longrightarrow \PP^2$ with building data as these, is such that $S$ is rational.

As usual, the pull--back to $S$ of the pencil $\mathcal F$ is a pencil of rational curves. So to prove the assertion, it suffices to prove that $S$ is regular. As usual, I blow--up the plane at $p$, getting $\pi: \mathbb F_1\longrightarrow \PP^2$, with the $(-1)$--curve $E$ of $\mathbb F_1$ contracted at $p$. Consider  the Cartesian diagram \eqref {eq:compos}, with $f': Y\longrightarrow \mathbb F_1$ a $\mathbb Z_2^2$--cover with building data  pull--back by $\pi$ of the building data of $f: S\longrightarrow \PP^2$. After normalizing $Y$ with the same notation as above, the branch data become
$$
\bar D_{10} =D'_{10}\sim aH-(a-1)E, \,\, \bar D_{01}= D'_{01}\sim bH-(b-1)E, \,\, \bar D_{11}=D'_{11}\sim cH-(c-1)E
$$
and accordingly 
$$
2L'_{10}\sim (a+c)H-(a+c-2)E, \,\, 2L'_{01}\sim (b+c)H-(b+c-2)E,\,\, 2L'_{11}\sim (a+b)H-(a+b-2)E.
$$
Now I notice that the branch curve $\bar D$ of $f': Y\longrightarrow \mathbb F_1$ may have at most double points or some triple point common to the three curves $D_{10}, D_{01}, D_{11}$. These points
 may correspond to singular points of $Y$. Then, either applying the results in \cite {Cat} or directly resolving the singularities, one easily sees that they do not drop the Euler characteristic of the structure sheaf of the resolution $Y'$ of $Y$.  Then a simple application of \eqref {eq:chi} shows that $\chi(\mathcal O_{Y'})=1$, proving the assertion. 

Let us now come to case (b). In this case we may have (up to reordering $D_{10}, D_{01}, D_{11}$) the following cases:\\
\begin{inparaenum}
\item [(b1)] $(\lambda, \mu, \nu)=(a,b-1,c-1)$;\\
\item [(b2)] $(\lambda, \mu, \nu)=(a,b,c-2)$.
\end{inparaenum}

I claim that  (b2) cannot happen. The proof is completely similar to the proof of the exclusion of cases (a2) and (a3)  and can be left to the reader. Case (b1) corresponds exactly to the situation described in the statement of the proposition and, with an argument I made already in case (a1), we see that this case can really happen and give rise to a rational $\mathbb Z_2^2$--cover of the plane. \end{proof}

\begin{remark}\label{rem:wer} The cases listed in the statement of Theorem \ref {prop:G'4}  correspond to the infinitely many cases of type {\bf (C.22)} in Theorem \ref {thm:class1}. Note that supposing the three curves $D_{10}, D_{01}, D_{11}$ to be irreducible, given a general point on each of it, its pull--back via $f$ consists of two points, so  there are three 
hyperelliptic curve of positive genera fixed by the action of the elements of $G$, as stated in Theorem \ref {thm:class1}, case {\bf (C.22)}. In case the three curves $D_{10}, D_{01}, D_{11}$ become reducible, the hyperelliptic curves accordingly degenerate. 
\end{remark}

\subsection{The case $r=3$} In this case $G\cong \mathbb Z_2^3$.

\subsubsection{The behaviour of $G'$ for $r=3$} We have a $\mathbb Z_2^3$--cover $f: S\longrightarrow \PP^2$ and there is (on the minimal resolution $S'$ of $S$) an invariant pencil $|F|$ by the $\mathbb Z_2^3$ actions that is mapped to the pencil $\mathcal F$ of lines through the point $p\in \PP^2$.

In the present situation, $G'$ (that is the group that acts on the general fibre of $|F|$), by the classification of finite abelian groups of $\PP^1$, can a priori either be $0$, or $\mathbb Z_2$ or $\mathbb Z_2^2$. However the former case cannot occur. In fact in that case $\pi(G)=G$ would be  isomorphic to $\mathbb Z_2^3$, and this is not possible, because $\pi(G)$ acts on the pencil $|F|$ that is a $\PP^1$. So there are only two possibilities:\\
\begin{inparaenum}
\item [(1)] $G'\cong \mathbb Z_2$ and $\pi(G)\cong \mathbb Z_2^2$;\\
\item [(2)] $G'\cong \mathbb Z_2^2$ and $\pi(G)\cong \mathbb Z_2$.
\end{inparaenum}

Looking at Theorem \ref {thm:class1}, we have the following cases to take care of:\\
\begin{inparaenum}[(i)]
\item cases in which $G'\cong \mathbb Z_2$, corresponding to {\bf (P1.2221)}, {\bf (P1.222)} and {\bf (C.2,22)} in Theorem \ref {thm:class1};\\
\item cases in which $G'\cong \mathbb Z_2^2$, corresponding to {\bf (P1s.222)} and {\bf (C.221)} in Theorem \ref {thm:class1}.
\end{inparaenum}

Let $\ell\in \mathcal F$ be a general line in the pencil. Accordingly we have the following possibilities:\\
\begin{inparaenum}[(i)]
\item the pull back of $\ell$ via $f$ consists of four distinct  curves in $|F|$ each mapping with degree 2 to $\ell$, in which case $G'\cong \mathbb Z_2$ and there are 2 ramification points on $\ell$, each to be counted with multiplicity 4;\\
\item the pull back of $\ell$ via $f$ consists of two distinct  curves in $|F|$ each mapping with degree 4 to $\ell$, in which case $G'\cong \mathbb Z_2^2$ and there are 3 ramification points on $\ell$, each to be counted with multiplicity 4. 
\end{inparaenum}

\subsubsection{The classification in the $G'=\mathbb Z_2$ case for $r=3$}

\begin{proposition}\label{prop:G'=0,3} Consider a $G$--cover $f: S\longrightarrow \PP^2$, with $G\cong \mathbb Z_2^3$, with $S$ normal and $G'\cong \mathbb Z_2$, presenting the invariant conic bundle case. Then, up to  a Cremona transformation, the image of the invariant pencil is the pencil $\mathcal F$ of lines through a point $p$ of the plane and the branch data of the cover are as follows:\\
\begin{inparaenum}[(i)]
\item  $D_{100}, D_{010},D_{001}$ are three distinct lines of $\mathcal F$;\\
\item   $D_{111}$ is a reduced curve of odd degree $d:=d_{111}\geq 1$, not containing any of the lines $D_{100}, D_{010},D_{001}$, and either $d=1$ and  $D_{111}$ does not contain the point $p$, or $d\geq 3$ and $D_{111}$ has a point of multiplicity $d-2$ at $p$;\\
\item $D_{110}, D_{101}, D_{011}$ are zero.
\end{inparaenum} 

Conversely, any $\mathbb Z_2^3$--cover $f: S\longrightarrow \PP^2$, with the above branch data and $G'=\mathbb Z_2$ is such that $S$ is rational and normal, presenting the invariant conic bundle case. \end{proposition}

\begin{proof}  Let us set $X=S/G'$ with the double cover $g: S\longrightarrow X$. The map $f: S\longrightarrow \PP^2$ factors through $g$ and a map $h: X\longrightarrow \PP^2$, that is a $\bar G$--cover, with $\bar G\cong \mathbb Z_2^2$ presenting the invariant conic bundle case, with $\bar G'=0$. Then by Proposition \ref {prop:G'=0}, up to a Cremona transformation we may assume that the branch curve of $h: S'\longrightarrow \PP^2$ consists of three distinct lines through the point $p$. These lines are also part of the branch curve of 
$f: S\longrightarrow \PP^2$. Since, as we saw, the general line $\ell$ through $p$ contains only two branch points, then the remaining part $D'$ of the branch curve is a curve of odd degree $d$ with multiplicity $d-2$ at $p$, or with multiplicity $d-1$ at $p$, in which case $p$ counts as a branch point on $\ell$. As in the proof of Proposition \ref {prop:G'2}, one sees that if $D'$ has multiplicity $d-1$ at $p$, up to a Cremona transformation of the plane, we may assume that $D'$ has degree 1 and does not pass through $p$.  From this analysis, the first part of the assertion follows. 

Suppose conversely that we have a $\mathbb Z_2^3$--cover $f: S\longrightarrow \PP^2$ with the branch data as in (i)--(iii) above. The proof is analogous to the end of the proof of Proposition \ref {prop:G'2}. The surface $S$ has a pencil $\mathcal G$ of rational curves, i.e., the pull--back via $f$ of the pencil of lines in $\mathcal F$. To finish the proof, we have to show that $\mathcal G$ is rational. 
Consider again  the double cover $g: S\longrightarrow X$. The image via $g$ of the invariant pencil, is a pencil $\mathcal F'$ of rational curves. The $\mathbb Z_2^2$--cover  $h: X\longrightarrow \PP^2$, that has branch curve $D_{100}+D_{010}+D_{001}$,
maps with degree 4 the pencil $\mathcal F'$ to $\mathcal F$, and the ramification occurs precisely at the three points corresponding to $D_{100}$, $D_{010}$ and $D_{001}$. This implies that $\mathcal F'$ is rational. Moreover, since $G'=\mathbb Z_2$, the pull--back via $g$ of a general curve in $\mathcal F'$ is a unique curve of the invariant pencil. This proves that $\mathcal G$ is rational as well as $\mathcal F'$, and this ends the proof.\end{proof}

\begin{remark}\label{rem:cases}  Cases {\bf (P1.222)} and {\bf (P1.2221)} appear in the statement of Proposition \ref {prop:G'=0,3} corresponding to the case $d=1$. Case {\bf (C.2,22)} corresponds to the case $d>3$. The hyperelliptic curves fixed by $G'$ corresponds to the $\mathbb Z_2^2$--cover $C$ of the curve $D_{111}$ that sits on $X$ and is the branch curve of the double cover $s: S\longrightarrow X$. Actually, if $D_{111}$ has genus $g'$ (if we assume that  $D_{111}$ has an ordinary singularity at $p$ and is elsewhere smooth, one has  $g'=d-2$, by Riemann--Hurwitz formula), $C$ has genus $g=4g'+3\equiv -1$ modulo 4 as predicted in Theorem \ref {thm:class1}, at the point {\bf (C.2,22)}. \end{remark}

\subsubsection{The classification in the $G'=\mathbb Z_2^2$ case for $r=3$}

\begin{proposition}\label{prop:G'=2,3} Consider a $G$--cover $f: S\longrightarrow \PP^2$, with $G\cong \mathbb Z_2^3$, with $S$ normal and $G'\cong \mathbb Z_2^2$, presenting the invariant conic bundle  case. Then, up to  a Cremona transformation, the image of the invariant pencil is the pencil $\mathcal F$ of lines through a point $p$ of the plane. The branch data of the cover are:\\
\begin{inparaenum}[(i)]
\item  $D_{100}, D_{010}, D_{110}$  curves of degrees $a,b,c$ (with the same parity) respectively, with a point of multiplicity $a-1, b-1, c-1$ respectively at $p$ (for at most one of these curves $p$ can be of multiplicity one more);\\
\item   $D_{001}$ that is the union of $2$ distinct lines through the point $p$ with no common component with $D_{100}, D_{010}, D_{110}$;\\
\item $D_{111}, D_{101}, D_{011}$  are zero.
\end{inparaenum} 

Conversely, given any $\mathbb Z_2^3$--cover $f: S\longrightarrow \PP^2$, with the above branch data and $G'=\mathbb Z_2^2$, then $S$ is rational and normal, presenting the invariant conic bundle  case. \end{proposition}

\begin{proof} Let $k\in \mathbb Z_2^3$ be an element not in $G'$. Let us set $T=S/\langle k\rangle$ with the double cover $g: S\longrightarrow T$. The map $f: S\longrightarrow \PP^2$ factors through $g$ and a map $h: T\longrightarrow \PP^2$, that is a $\bar G$--cover, with $\bar G\cong \mathbb Z_2^2$ presenting the invariant conic bundle  case, with $\bar G'\cong \bar G\cong \mathbb Z_2^2$. Note that the image on $T$ of the invariant pencil of rational curves $|F|$ on $S$ is a pencil $\mathcal F'$ of rational curves that maps via $h$ to the pencil of lines $\mathcal F$. Since $h: T\longrightarrow \PP^2$,  is a $\bar G$--cover, with $\bar G\cong \mathbb Z_2^2$ with $\bar G'\cong \bar G\cong \mathbb Z_2^2$, then the map of $\mathcal F'$ to $\mathcal F$ is birational, hence $\mathcal F'$ is also rational. 

By Proposition \ref {prop:G'4}, up to a Cremona transformation we may assume that the branch curve of $h: T\longrightarrow \PP^2$ consists of three distinct curves of degree $a,b,c$ respectively of the same parity, with a point of multiplicity $a-1,b-1,c-1$ at $p$ respectively (for at most one of these curves $p$ can be of multiplicity one more). These curves are also part of the branch curve of 
$f: S\longrightarrow \PP^2$. Since, as we saw, the general line $\ell$ through $p$ contains only three branch points, then the remaining part of the branch curve is a 
union of $d$ lines $\ell_1,\ldots, \ell_d$ through $p$.  I claim that $d=2$. Indeed, the union of the pull back of $\ell_1,\ldots, \ell_d$ to $T$ is the branch curve of the double cover $g: S\longrightarrow T$, that therefore consists of $d$ rational curves in a rational pencil. Since $S$ is rational, one immediately has $d=2$. 

Suppose conversely that we have a $\mathbb Z_2^3$--cover $f: S\longrightarrow \PP^2$ with the branch data as in (i)--(iii) above. The same argument I made above shows that $S$  has a rational pencil of rational curves, hence it is rational, proving the proposition. \end{proof}

\begin{remark}\label{rem:frot} The cases in Proposition \ref {prop:G'=2,3} correspond to cases {\bf (P1s.222)} and {\bf (C.221)} in Theorem \ref {thm:class1}.\medskip

(i) As for case  {\bf (P1s.222)}, it corresponds to the case in which $a=b=c=1$ and the three lines $D_{100}, D_{010}, D_{110}$ have a common point $q\neq p$. The branch curve $D$ is then the union of the three lines $D_{100}, D_{010}, D_{110}$  and the two lines $D_{001}$ passing through the point $q\neq p$. It is interesting to look more closely to this $\mathbb Z_2^3$--cover of $\PP^2$.  Let us blow--up the two points $p,q$, thus obtaining a morphism $\pi: P\longrightarrow \PP^2$, with two exceptional divisors $E, E'$ contracted to $p,q$ respectively. Let us denote by $H$ the strict transform on $P$ of a general line on $\PP^2$. By pulling back the $\mathbb Z_2^3$--cover of $\PP^2$ via $\pi$ and after normalizing, we have a $\mathbb Z_2^3$--cover $f': Y\longrightarrow P$ with branch data
$$
D'_{100}\sim D'_{010} \sim D'_{110}\sim H-E',\,\, D'_{001}\sim 2H-2E
$$
where, as usual, I denote by $D'_{100},D'_{010},D'_{110},D'_{001}$ the strict transform of $D_{100},D_{010},$ $D_{110},D_{001}$ to $P$. By applying \eqref {eq:K2}, we compute $K_Y^2=0$. 

Next look at the two pencils $|H-E|$ and $|H-E'|$. The pull back of the general member $N$ of $|H-E|$ is the union of two disjoint rational curves $A+A'$ that are two $\mathbb Z^2_2$--cover of $\PP^1$, each ramified at the three points in which $N$ intersects the three curves $D'_{100},D'_{010},D'_{110}$. As we saw, $A$ and $A'$ are linearly equivalent and belong to a rational pencil of rational curves $|A|$.  
Similarly, the pull back of the general member $M$ of $|H-E'|$ is the union of four disjoint rational curves $B_1+\cdots +B_4$ each being a double cover of $\PP^1$, each ramified at the two points in which $M$ intersects the three curves in $D'_{001}$. Again, $B_1,\ldots, B_4$ belong to the same rational pencil $|B|$ of rational curves. Now we  have $8=(2A)\cdot (4B)$ hence $A\cdot B=1$ and therefore we have a birational morphism $g:Y\longrightarrow \PP^1\times \PP^1$ such that the pull back of the two natural rulings of $\PP^1\times \PP^1$ are the pencils $|A|$ and $|B|$. 

Notice that that the pull back of $E'$ via $f'$ is the union of two disjoint rational curves $E'_2, E'_2$ both being  a $\mathbb Z^2_2$--cover of $\PP^1$ branched at the three points in which the curves $D'_{100},D'_{010},D'_{110}$ intersect $E'$. One has $(E'_2)^2= (E'_2)^2=-4$. Similarly the pull back of $E$ via $f'$ is the union of four disjoint rational curves $E_1,\dots, E_4$ each being  a double cover of $\PP^1$ branched at the two points in which the curve $D'_{001}$ intersects $E$. One has $(E_i)^2=-2$, for $1\leq i\leq 4$.

The unique curve $C\in |H-E-E'|$ is a $\PP^1$ and $C^2=-1$ and $f'$ has no ramification point on it, so the pull back of $C$ via $f'$ is the union of eight distinct copies of $\PP^1$, each with self--intersection $-1$.
 The pencil $|A|$ contains the two curves consisting of $E'_2, E'_2$ plus four of the $(-1)$--curves over $C$ each intersecting $E'_2, E'_2$ in one point. Similarly $|B|$ contains the four curves consisting of  $E_1,\dots, E_4$ plus two of the $(-1)$--curves over $C$ each intersecting $E_1,\dots, E_4$ in one point.

The morphism  $g:Y\longrightarrow \PP^1\times \PP^1$ is nothing but the contraction of the eight $(-1)$--curves over $C$.

I note also that in this case one can simplify the branch curve in the following way. Consider the intersection point $r$ of the line $D_{100}$ with one of the lines 
forming $D_{001}$ and make the quadratic transformation based at $p,q,r$. This quadratic transformation contracts to points the line $D_{100}$ and the line of $D_{001}$ passing through $r$. Then it blows up to lines the points $p,q,r$. However, the first two do not contribute to the branch curve whereas the third does. So after this quadratic transformation the branch curve consists of only four lines 
$D_{001}, D_{110}, D_{010}, D_{101}$. Then one computes the building data that are all linearly equivalent to $\mathcal O_{\PP^2}(1)$ except $L_{011}\cong \mathcal O_{\PP^2}(2)$. An application of \eqref {eq:chi} and \eqref {eq:K2} shows that for the corresponding $\mathbb Z_2^3$--cover $f: S\longrightarrow \PP^2$ one has $\chi(\mathcal O_S)$ and $K^2_S=8$. 

It is easy to check that in fact $S$ is a quadric. To see this, consider the morphism
$$
\psi: (x:y:z:t)\in \PP^3\longrightarrow (x^2:y^2:z^2:t^2)\in \PP^3
$$
that is a $\mathbb Z_2^3$--cover of $\PP^3$ branched over the coordinate planes. Consider the smooth quadric $S$ with equation $x^2+y^2+z^2+t^2=0$. The image of $S$ via $\psi$ is the plane with  equation $x+y+z+t=0$. Hence the restriction of $\psi$ to $S$ determines a $\mathbb Z_2^3$--cover $f: S\longrightarrow \PP^2$ that is branched along four general lines, as above. 
\medskip

(ii) As for the case {\bf (C.221)}, look at the three curves $D_{100}, D_{010}, D_{110}$ and suppose they are irreducible and in a general position, in particular they do not pass through a same point $q$ other than $p$. For each one of them, the counterimage via $f$ of the general point consist of four points. The branch points of these fourthuple covers occur at the pairwise intersection of these curves and at the intersection with the two lines in $D_{001}$, where the counterimage via $f$ of the points consist of only two points. 
Computing the genus $g$ of the cover of each one of the curves $D_{100}, D_{010}, D_{110}$, say of $D_{100}$, by Riemann--Hurwitz formula one finds
$$
8+2g-2=4+2(2a+b+c-2)
$$
so that $g=2a+b+c-3$, that is odd, because $a,b,c$ have the same parity. It is also not difficult to see that these fourthuple covers of $D_{100}, D_{010}, D_{110}$ are in fact hyperelliptic curves, as stated in  Theorem \ref {thm:class1}, case {\bf (C.221)}. 
\end{remark}

\subsection{The case $r=4$} In this case $G\cong \mathbb Z_2^4$.

\subsubsection{The behaviour of $G'$ for $r=4$} We have a $\mathbb Z_2^4$--cover $f: S\longrightarrow \PP^2$ and there is (on the minimal desingularization $S'$ of $S$) an invariant pencil $|F|$ by the $\mathbb Z_2^3$ actions that is mapped to the pencil $\mathcal F$ of lines through the point $p\in \PP^2$.

In the present situation, $G'$ (that is the group that acts on the general fibre of $|F|$), by the classification of finite abelian groups of $\PP^1$, can a priori either be $0$, or $\mathbb Z_2$ or $\mathbb Z_2^2$. However the former two cases cannot occur. In fact in that case $\pi(G)=G$ would be isomorphic to $\mathbb Z_2^s$, with $s\geq 3$, and this is not possible by the classification of finite abelian groups of $\PP^1$, because $\pi(G)$ acts on the pencil $|F|\cong \PP^1$. So there is only the possibility that $G'\cong \pi(G)\cong \mathbb Z_2^2$. 

Looking at Theorem \ref {thm:class1}, we have only two cases to take care of, namely  {\bf (P1.2222)} and {\bf (C.22,22)}. 

Let $\ell\in \mathcal F$ be a general line in the pencil. Then we have only the following possibility:  the pull back of $\ell$ via $f$ consists of four distinct  curves in $|F|$ each mapping with degree 4 to $\ell$ and there are 3 ramification points on $\ell$, each to be counted with multiplicity 8. 

\subsubsection{The classification in the case  $r=4$}

\begin{proposition}\label{prop:G'=2,4} Consider a $G$--cover $f: S\longrightarrow \PP^2$, with $G\cong \mathbb Z_2^4$, with $S$ normal, presenting the invariant conic bundle  case. Then, up to  a Cremona transformation, the image of the invariant pencil is the pencil $\mathcal F$ of lines through a point $p$ of the plane. The branch data of the cover are:\\
\begin{inparaenum}[(i)]
\item  $D_{1000}, D_{0100}, D_{1100}$ are curves of degrees $a,b,c$ (with the same parity) respectively, with a point of multiplicity $a-1, b-1, c-1$ respectively at $p$ (and only one of them can have multiplicity one more at $p$);\\
\item three more distinct lines $D_{0010}, D_{0001}, D_{0011}$ containing $p$, not contained in $D_{1000}, D_{0100}$, $D_{1100}$;\\
\item the remaining branch data are zero.
\end{inparaenum} 

Conversely, any $\mathbb Z_2^4$--cover $f: S\longrightarrow \PP^2$, with the above branch data is such that $S$ is rational and normal, presenting the invariant conic bundle case. \end{proposition}

\begin{proof} The proof imitates the one of Proposition \ref {prop:G'=2,3}.

The sequence \eqref {eq:split} splits. So we have $\pi(G)\subset G$. Let us set $T=S/\pi(G)$ with the $\mathbb Z_2^2$--cover $g: S\longrightarrow T$. The map $f: S\longrightarrow \PP^2$ factors through $g$ and a map $h: T\longrightarrow \PP^2$, that is a $\bar G$--cover, with  $\bar G\cong \mathbb Z_2^2$ presenting the invariant conic bundle structure case, with $\bar G'\cong \bar G\cong \mathbb Z_2^2$. 

Then by Proposition \ref {prop:G'4}, up to a Cremona transformation we may assume that the branch curve of $h: T\longrightarrow \PP^2$ consists of three distinct curves of degree $a,b,c$ respectively of the same parity, with a point of multiplicity $a-1,b-1,c-1$ at $p$ respectively (and only one of them can have multiplicity one more at $p$). These curves are also part of the branch curve of  $f: S\longrightarrow \PP^2$. Since, as we saw, the general line $\ell$ through $p$ contains only three branch points, then the remaining part of the branch curve is a  union of $d$ lines $\ell_1,\ldots, \ell_d$  through $p$.  

The pull--back on $T$ of a general line $\ell$ in $\mathcal F$ is a unique rational curve that is a $\mathbb Z_2^2$--cover of $\ell$. So on $T$ there is a rational pencil $\mathcal F'$ of rational curves, i.e., the pull back of $\mathcal F$.  

The pull back of $\ell_1,\ldots, \ell_d$ are $d$ curves in the pencil $\mathcal F'$ that are  the branch curve of the $\mathbb Z_2^2$--cover $g: S\longrightarrow T$. The pull back via $f$ of a general line $\ell\in \mathcal F$ is the union of 4 rational curves (each belonging to the invariant pencil $|F|$), each being a $\mathbb Z_2^2$--cover of $\ell$. This implies that  the map $g: S\longrightarrow T$ induces a $\mathbb Z_2^2$--cover of the invariant pencil $|F|$ to $\mathcal F'$, that has therefore exactly 3 branch points. This implies that $d=3$ proving that the branch data are as in (i)--(iii). 

Suppose conversely that we have a $\mathbb Z_2^4$--cover $f: S\longrightarrow \PP^2$ with the branch data as in (i)--(iii) above. The same argument I made above shows that $S$  has a rational pencil of rational curves, hence it is rational, proving the proposition. \end{proof}

\begin{remark}\label{rem:klip} (i) As a sanity check, suppose we have a $\mathbb Z_2^4$--cover $f: S\longrightarrow \PP^2$ with the branch data as in (i)--(iii) of Proposition \ref {prop:G'4}. Then $S$ has a pencil of rational curves.  Let us check that $S$ is regular, hence rational. Suppose that we are in a \emph{general situation}, namely the curves $D_{1000}, D_{0100}, D_{1100}$ have ordinary singularities at $p$ and intersect transversely off $p$ (otherwise the proof becomes slightly more complicated but is conceptually the same). Let us blow--up the plane at $p$, getting $\pi: \mathbb F_1\longrightarrow \PP^2$, with the $(-1)$--curve $E$ of $\mathbb F_1$ contracted at $p$. Consider again the Cartesian diagram \eqref {eq:compos}, with $f': Y\longrightarrow \mathbb F_1$ a $\mathbb Z_2^4$--cover with building data just pull--back by $\pi$ of the building data of $f: S\longrightarrow \PP^2$. After normalizing $Y$ to a surface $S'$ the branch data become (with the usual notation)
$$
\begin{array}{c}
\bar D_{1000}\sim aH-(a-1)E,\,\, \bar D_{0100}\sim bH-(b-1)E,\,\, \bar D_{1100}\sim cH-(c-1)E,\\
\bar D_{0010}\sim \bar D_{0001}\sim \bar D_{0011}\sim H-E
\end{array}
$$
and all the other branch data are zero. At this point using \eqref 
{eq:data}, one computes the building data and checks that $a,b,c$ must have the same parity. Then, by applying  \eqref {eq:chi}, one sees  that $\chi(\mathcal O_{S'})=1$.\medskip

(ii) The cases in Proposition \ref {prop:G'=2,4} are related to cases
{\bf (P1.2222)} and  {\bf (C.22,22)} of Theorem \ref {thm:class1}. As for case {\bf (P1.2222)}, the situation is not so different from the one in the case {\bf (P1s.222)} described in Remark \ref {rem:frot}, (i). It corresponds to the case in which $a=b=c=1$ and the three lines $D_{1000}, D_{0100}, D_{1100}$ have a common point $q\neq p$. The branch curve $D$ is then the union of the three lines $D_{1000}, D_{0100}, D_{1100}$ passing through the point $q$ and the three lines $D_{0010}, D_{0001}, D_{0011}$ passing through the point $p$. 

Let us blow--up the two points $p,q$, thus obtaining a morphism $\pi: P\longrightarrow \PP^2$, with two exceptional divisors $E, E'$ contracted to $p,q$ respectively. Let us denote by $H$ the strict transform on $P$ of a general line on $\PP^2$. By pulling back the $\mathbb Z_2^4$--cover of $\PP^2$ via $\pi$ and after normalizing, we have a $\mathbb Z_2^4$--cover $f': Y\longrightarrow P$ with branch data
$$
D'_{1000}\sim D'_{0100} \sim D'_{1100}\sim H-E',\,\, D'_{0010}\sim D'_{0001} \sim D'_{0011}\sim H-E,
$$
(the notation is the usual one).  By applying \eqref {eq:K2}, we compute $K_Y^2=-8$. 

Look at the two pencils $|H-E|$ and $|H-E'|$. The pull back of the general member of $|H-E|$ [resp. of $|H-E'|$] is the union of four  disjoint rational curves belonging to rational pencil of rational curves $|A|$  [resp.  $|B|$] Now we  have $16=(4A)\cdot (4B)$ hence $A\cdot B=1$ and therefore we have a birational morphism $g:Y\longrightarrow \PP^1\times \PP^1$ such that the pull back of the two natural rulings of $\PP^1\times \PP^1$ are the pencils $|A|$ and $|B|$. 

The pull back of $E$ [resp. of $E'$] via $f'$ is the union of four disjoint rational curves each being  a $\mathbb Z^2_2$--cover of $\PP^1$ branched at the three points in which the curves $D'_{0010}, D'_{0001},  D'_{0011}$ [resp. the curves  $D'_{1000}, D'_{0100}, D'_{1100}$] intersect $E$ [resp. intersect $E'$]. The self--intersection of each such curve is $-4$. 

The unique curve $C\in |H-E-E'|$ is a $\PP^1$ and $C^2=-1$ and $f'$ has no ramification point on it, so the pull back of $C$ via $f'$ is the union of sixteen distinct copies of $\PP^1$, each with self--intersection $-1$. The pencil $|A|$ [resp. $|B|$] contains the four curves over $E'$ [resp. over $E$] each plus four of the $(-1)$--curves over $C$ each intersecting the curves over $E'$ [resp. over $E$] in one point. The morphism  $g:Y\longrightarrow \PP^1\times \PP^1$ is nothing but the contraction of the sixteen  $(-1)$--curves over $C$.\medskip

(iii) As in the case {\bf (C.22,22)}, again the situation is similar to the one in Remark \ref {rem:frot}, (ii).
Look at the three curves $D_{1000}, D_{0100}, D_{1100}$ and suppose they are irreducible and in a general position, in particular they do not pass through a same point $q$ other than $p$. For each one of them, the counterimage via $f$ of the general point consist of 8  points. The branch points of these 8--tuple covers occur at the pairwise intersection of these curves and at the intersections with the the lines in $D_{0010}, D_{0001}, D_{0011}$, where the counterimage via $f$ of the points consist of only 4 points. So each branch point counts with multiplicity 4.
Computing the genus $g$ of the cover of each one of the curves $D_{1000}, D_{0100}, D_{1100}$, say of $D_{1000}$, by Riemann--Hurwitz formula one finds
$$
16+2g-2=12+4(2a+b+c-2)
$$
so that $g=2(2a+b+c)-1$, that is congruent to 3 modulo 4. It is also not difficult to see that these 8--tuple covers of $D_{1000}, D_{0100}, D_{1100}$ are in fact hyperelliptic curves, as stated in  Theorem \ref {thm:class1}, case {\bf (C.221)}. 
\end{remark}

\section{The classification in the Del Pezzo case}\label{sec:classinva}

In this section I consider a $G$--cover $f: S\longrightarrow \PP^2$, with $G\cong \mathbb Z_2^r$, with $S$ rational and normal (and $r>1$), and  I will consider the Del Pezzo case. Again, as we saw in Theorem \ref {thm:class1}, we have $2\leq r\leq 4$. I will consider the minimal desingularization $\phi: S'\longrightarrow S$. I will assume that the pair $(S', \mathbb Z_2^r)$ is a minimal $\mathbb Z_2^r$--surface and that ${\rm Pic}(S')^{\mathbb Z_2^r}\cong \mathbb Z$ generated over $\mathbb Q$ by the class of $K_{S'}$ (see Proposition \ref {prop:bla}). I will keep this notation throughout this section.

\subsection{The classification in the $r=2$ case} In the Del Pezzo case with $r=2$, according to Theorem  \ref {thm:class2}, we have to take care of the cases  {\bf (2.G2)} and {\bf (1.B2.1)}. 

\subsubsection{The degree 2 Del Pezzo case} In the case {\bf (2.G2)}, the $\mathbb Z_2^2$--cover $f: S\longrightarrow \PP^2$ is such that $S'$ is a Del Pezzo surface of degree 2 that is a double cover $g: S'\longrightarrow \PP^2$ (via the Bertini involution $\sigma$), branched along a smooth quartic curve $\Gamma$ that is invariant via a non--trivial involution $\iota$ of the plane, and $\mathbb Z_2^2=\langle \sigma, \iota\rangle$. 

\begin{proposition}\label{prop:bert}  In the case {\bf (2.G2)}, the $\mathbb Z_2^2$--cover $f: S\longrightarrow \PP^2$ has, up to a Cremona transformation of the plane, the following branch data:\\
\begin{inparaenum}[(i)]
\item $D_{10}$ is a quartic curve with a tacnode at a point $x$;\\
\item $D_{01}$ is a smooth conic passing through $x$  and having at $x$ the same 
 tangent as the tacnodal tangent of $D_{10}$;\\
 \item $D_{11}=0$.
\end{inparaenum}

Conversely, any $\mathbb Z_2^2$--cover $f: S\longrightarrow \PP^2$ with the above branch data presents the Del Pezzo case {\bf (2.G2)}.
\end{proposition}

\begin{proof} We start with a double cover $g: S'\longrightarrow \PP^2$ with $S'$ a Del Pezzo surface of degree 2, $g$ given by  the Bertini involution $\sigma$, branched along a smooth quartic curve $\Gamma$. Suppose  that $\Gamma$ is invariant by an involution $\iota$ of the plane. The involution $\iota$ has a fixed point $p\not\in \Gamma$ and a line of fixed points $\ell$ that does not contain $p$. Consider the quotient $Q=\PP^2/\langle \iota\rangle$ and the double cover $h: \PP^2\longrightarrow Q$. The surface $Q$ is isomorphic to a quadric cone in $\PP^3$ with vertex $q$  and the double cover $h: \PP^2\longrightarrow Q$ is branched at $q$ (corresponding to the point $p$) and along a smooth conic $\Lambda$ corresponding to the line $\ell$. Let $\Delta$ be the image of $\Gamma$ on $Q$. Note that $h: \Gamma \longrightarrow \Delta$ has  ramification points at the intersections of $\Gamma$ with $\ell$, hence $\Delta$ is a quartic curve with arithmetic genus 1, hence it is the complete intersection of $Q$ with a quadric not containing $q$. The composite map $h\circ g: S'\longrightarrow Q$ is a $\mathbb Z_2^2$--cover. To get a model of our original $\mathbb Z_2^2$--cover $f: S\longrightarrow \PP^2$, we need to birationally project $Q$ to the plane  from a general point of $Q$. In doing so we have the branch data as stated in (i)--(iii). 

Conversely, assume we have a $\mathbb Z_2^2$--cover $f: S\longrightarrow \PP^2$, with the branch data as stated in (i)--(iii) and let us prove that $S$ is rational and in fact it is related to  case {\bf (2.G2)}. We will assume that the quartic $D_{10}$ and the conic $D_{10}$ have intersection multiplicity exactly 4 at $x$ and intersect transversally in four points off $x$. If not the proof is analogous (and only slightly more involved) and can be left to the reader. 

The building data here are
$$
L_{10}\cong \mathcal O_{\PP^2}(2), \,\, L_{01}\cong \mathcal O_{\PP^2}(1),\,\, L_{11}\cong \mathcal O_{\PP^2}(3). 
$$

The surface $S$ is singular over the point $x$. We want to desingularize it. First I blow up $x$ and get $\pi: \mathbb F_1\longrightarrow \PP^2$, with the $(-1)$--curve $E$ of $\mathbb F_1$ contracted at $x$. Consider  the Cartesian diagram \eqref {eq:compos}, with $f': Y\longrightarrow \mathbb F_1$ a $\mathbb Z_2^2$--cover with building data the pull--backs via $\pi$ of the building data of $f: S\longrightarrow \PP^2$. Let $D'_{10}, D'_{01}$ be the strict transforms of $D_{10}, D_{01}$ on $\mathbb F_1$. The branch data of $f': Y\longrightarrow \mathbb F_1$ are 
$$\bar D_{10}=D'_{10}+2E, \bar D_{01}=D'_{01}+E, \bar D_{11}=0,$$
 so that $Y$ is not normal. Let us normalize it with the recipe described in Section \ref {sss:norm}. The normalization $Z$ determines a $\mathbb Z_2^2$--cover  $f'': Z\longrightarrow \mathbb F_1$ with branch data 
 $$\bar D'_{10}=D'_{10}, \bar D'_{01}=D'_{01}+E, \bar D'_{11}=0.$$
  The normalization $Z$ is still singular, since the total branch curve $\bar D'=\bar D'_{10}+\bar D'_{01}$ has a quadruple point at the intersection point $y$ of $E$ with $D'_{10}$ and $D'_{01}$ (because $D'_{10}$ has a node on $E$ and $D'_{01}$ passes simply through that node). 
At this point I could apply the results in \cite {Cat} to compute the invariants of the minimal resolution of $Z$ and prove that it is rational. However, to be self--contained, I proceed as above desingularizing $Z$. 

Let us  blow up $y$ getting a morphism $\pi': \tilde {\mathbb F}_1\longrightarrow \mathbb F_1$, with exceptional curve $E'$ contracted at $y$. The strict transform of $E$ is a $(-2)$--curve that I still denote by $E$. By pulling back the $\mathbb Z^2_2$--cover $f'': Z\longrightarrow \mathbb F_1$ and normalizing, we get a new $\mathbb Z^2_2$--cover $f''': T\longrightarrow \tilde {\mathbb F}_1$. Let us look at the branch data of $f''': T\longrightarrow \tilde {\mathbb F}_1$. These are $\tilde D_{10}$, the strict transform of $D'_{10}$ on $\tilde {\mathbb F}_1$, and $\tilde D_{01}$ that is the sum of the strict transform of $D'_{01}$ with $E$, whereas $\tilde D_{11}=0$. Since the total branch curve has only normal crossing singularities, $T$ is smooth.

Let us denote by $H$ the pull--back on $\tilde {\mathbb F}_1$ of a general line of $\PP^2$. Then we have
$$
\tilde D_{10}\cong 4H-2E-4E', \,\, \tilde D_{01}\cong 2H-2E'
$$
so that the  building data are 
$$
\tilde L_{10}\cong 2H-E-2E', \,\, \tilde L_{01}\cong H-E', \,\, \tilde L_{11}\cong 3H-E-3E'.
$$

Now I claim that 
$$
h^1(\tilde {\mathbb F}_1, \tilde L_{10}^{-1})=h^1(\tilde {\mathbb F}_1, \tilde L_{01}^{-1})=
h^1(\tilde {\mathbb F}_1, \tilde L_{11}^{-1})=0.
$$

The proof that $h^1(\tilde {\mathbb F}_1, \tilde L_{10}^{-1})=0$ is trivial because $\tilde L_{10}\cong 2H-E-2E'$ is clearly big and nef. As for $h^1(\tilde {\mathbb F}_1, \tilde L_{01}^{-1})$, this equals $h^1(\tilde {\mathbb F}_1, K_{\tilde {\mathbb F}_1}+\tilde L_{01})$. Now 
$$
K_{\tilde {\mathbb F}_1}+\tilde L_{01}\cong -3H+E+2E'+H-E'=-2H+E+E'
$$
and therefore
$$
h^1(\tilde {\mathbb F}_1, \tilde L_{01}^{-1})=h^1(\tilde {\mathbb F}_1, -2H+E+E')
$$ 
that is 0 because $2H-E-E'$ is big and nef. 

As for $h^1(\tilde {\mathbb F}_1, \tilde L_{11}^{-1})$, this equals $h^1(\tilde {\mathbb F}_1, K_{\tilde {\mathbb F}_1}+\tilde L_{11})$. Now 
$$
K_{\tilde {\mathbb F}_1}+\tilde L_{11}\cong -3H+E+2E'+3H-E-3E'=-E'
$$
and therefore
$$
h^1(\tilde {\mathbb F}_1, \tilde L_{01}^{-1})=h^1(\tilde {\mathbb F}_1, -E')=0.
$$ 
Indeed, consider the exact sequence 
$$
0\longrightarrow \mathcal O_{\tilde {\mathbb F}_1}(-E')\longrightarrow \mathcal O_{\tilde {\mathbb F}_1}\longrightarrow \mathcal O_{E'}\longrightarrow 0.
$$
The map $H^0(\tilde {\mathbb F}_1, \mathcal O_{\tilde {\mathbb F}_1})\longrightarrow H^0(E', \mathcal O_{E'})$ is surjective and $h^1(\tilde {\mathbb F}_1, \mathcal O_{\tilde {\mathbb F}_1})=0$, that implies that $h^1(\tilde {\mathbb F}_1, -E')=0$.

In conclusion, by taking into account \eqref {eq:reg}, we see that $T$ has irregularity 0. 

Next, by Hurwitz formula, we have that 
\begin{equation}\label{eq:sc}
2K_T\sim (f''')^*(2K_{\tilde {\mathbb F}_1}+\tilde D)\sim (f''')^*\Big (2(-3H+E+2E')+(6H-2E-6E')\Big)= -2 (f''')^*(E')
\end{equation}
whence clearly $h^0(T, 2K_T)=0$. Then by Castelnuovo's rationality criterion $T$ is rational, ending our proof. 
\end{proof}

\begin{remark}\label{rem:sc} (i) Let us keep the notation of the proof of Proposition \ref {prop:bert}.  By \eqref {eq:sc} or \eqref {eq:K2}, we see that $K_T^2=-4$. On the other hand we know that $T$ is a modification of a Del Pezzo surface $S'$ of degree 2. Precisely, $T$ is the blow--up of $S'$ at six points. In fact, remember that we started with a double cover $g: S'\longrightarrow \PP^2$ with $S'$ a Del Pezzo surface of degree 2, $g$ given by  the Bertini involution $\sigma$, branched along a smooth quartic curve $\Gamma$ that is invariant by an involution $\iota$ of the plane. As we know, $\iota$ has a fixed point $p$ off $\Gamma$. Then we 
considered the quotient $Q=\PP^2/\langle \iota\rangle$ that is isomorphic to a quadric cone in $\PP^3$ with vertex $q$ (the image of $p$) and there is the double cover $h: \PP^2\longrightarrow Q$. We have then the $\mathbb Z_2^2$--cover $h\circ g: S'\longrightarrow Q$. To get a model of the original $\mathbb Z_2^2$--cover $f: S\longrightarrow \PP^2$, we need to birationally project $Q$ to the plane  from a general point $q$ of $Q$. Now it is immediate to see that $T$ is the blow--up of $S'$ at the two points pull--back of $p$ via $g$ and at the four points pull--back of $q$  via $h\circ g$. On the whole, we find $K_{S'}^2=2$, as it should be.\smallskip

(ii) One can slightly simplify the branch curve in Proposition \ref {prop:bert} in the following two equivalent ways. One way is to project down the quadric cone $Q$ from an intersection point $t$ of $\Lambda$ with $\Delta$. Then one sees that the branch data become $D'_{10}$ a cubic curve, $D'_{01}$ a  line, $D'_{11}$ a line tangent to the cubic curve $D'_{10}$. Another way to see this is to make a quadratic transformation based at $x$, at the infinitely near point $y$ to $x$ along the conic $D_{01}$ and at an intersection point $z$ of $D_{10}$ and $D_{01}$ off $x$ and $y$.  The quartic $D_{10}$ maps to a cubic $D'_{10}$, the conic $D_{01}$ maps to a line $D'_{01}$, the point $z$ blows--up to a line $D'_{11}$ that is tangent to $D'_{10}$. 

If we consider the $\mathbb Z_2^2$--cover with these branch data, this is still singular over the tangency point of $D'_{11}$ with $D'_{10}$. After desingularizing (for this one has to blow up the plane at the tangency point of $D'_{11}$ with $D'_{10}$ and at the infinitely near point that $D'_{11}$ and $D'_{10}$ still have in common), one finds a rational surface $X$ such that $K_X^2=-1$. This is no contradiction with (i) above, because in this case $X$ is obtained by the degree 2 Del Pezzo surface $S'$ by blowing up only three points (not six as before), namely the two points of $S'$ over $p\in \PP^2$ and the only point of $S'$ over the point $t$. 

 \end{remark}

\subsubsection{The degree 1 Del Pezzo case} In the case  {\bf (1.B2.1)}, the $\mathbb Z_2^2$--cover $f: S\longrightarrow \PP^2$ is such that $S$ is a Del Pezzo surface of degree 1 that is a double cover $g: S\longrightarrow Q$ of a quadric cone $Q$ in $\mathbb P^3$ (via the Geiser involution $\tau$), branched at the vertex of the cone and along a smooth sextic  curve $\Gamma$ of genus 4 (the complete intersection of $Q$ with a cubic surface),  and $\Gamma$ is invariant via a non--trivial involution $\iota$ of $Q$, and $\mathbb Z_2^2=\langle \tau, \iota\rangle$. We have a double cover $h: Q\longrightarrow Q/\langle \iota \rangle\cong \PP^2$ and $f=h\circ g$. 

\begin{proposition}\label{prop:geiser}  In the case {\bf (1.B2.1)}, the $\mathbb Z_2^2$--cover $f: S\longrightarrow \PP^2$ has, up to a Cremona transformation of the plane, the following branch data:\\
\begin{inparaenum}[(i)]
\item $D_{10}$ and $D_{01}$ are two distinct lines intersecting at a point $p$ of the plane;\\
\item $D_{11}$ is a smooth cubic curve  that does not pass through the point $p$ and intersects transversely the lines $D_{10}$ and $D_{01}$.
\end{inparaenum}

Conversely, any $\mathbb Z_2^2$--cover $f: S\longrightarrow \PP^2$ with the above branch data presents  the Del Pezzo case {\bf (1.B2.1)}.
\end{proposition}

\begin{proof} In the double cover $h: Q\longrightarrow Q/\langle \iota \rangle\cong \PP^2$ the branch curve consists of two distinct lines intersecting at a point $p$ of the plane. These two lines belong to the branch curve $D$ of the $\mathbb Z_2^2$--cover $f: S\longrightarrow \PP^2$. The only other component of the branch curve is the image on the plane of the curve $\Gamma$. The restriction of $h$ to $\Gamma$ is a double cover $h_{|\Gamma}: \Gamma \longrightarrow \Gamma'$, that has 6 branch points, i.e., the intersections of $\Gamma$ with the two lines that are the branch curve of $h$. This implies that $\Gamma'$ is a degree three curve of genus 1. Then we see that  the branch data verify (i)--(ii).

Conversely, suppose that  we have a  $\mathbb Z_2^2$--cover $f: S\longrightarrow \PP^2$ with the  branch data given by (i) and (ii). Then $S$ is smooth and we have
the building data
$$
L_{10}\cong L_{01}\cong \mathcal O_{\PP^2}(2), \,\, L_{11}\cong \mathcal O_{\PP^2}(1).
$$
and applying \eqref {eq:chi} and \eqref {eq:K2} we get $\chi(\mathcal O_S)=1$ and $K_S^2=1$. It remains to prove that $S$ is rational and actually a Del Pezzo surface. 

First of all $S$ is regular, by \cite [Prop. 2.13, (3)]{CCT}. Next we have
$$
2K_S\sim f^*(2K_{\PP^2}+D)\sim f^*(-6H+5H)=f^*(-H)
$$
that proves that the bigenus of $S$ is $0$ and the rationality of $S$ follows. Finally it is clear that $S$ is a Del Pezzo surface. 
\end{proof}

\subsection{The classification in the $r=3$ case} In the Del Pezzo case with $r=3$, according to Theorem  \ref {thm:class2}, we have to take care of the cases  {\bf (4.222)} and {\bf (2.G22)}. 

\subsubsection{The degree 4 Del Pezzo case} In the case  {\bf (4.222)}, the $\mathbb Z_2^3$--cover $f: S\longrightarrow \PP^2$ arises from  a Del Pezzo surface $S'$ of degree 4 given by equations \eqref {eq:dp4} and the group $G=\mathbb Z_2^3$ is given by \eqref {eq:GG}.

\begin{proposition}\label{prop:geiserol}  In the case {\bf (4.222)}, the $\mathbb Z_2^3$--cover $f: S\longrightarrow \PP^2$ has, up to a Cremona transformation of the plane, the following branch data:\\
\begin{inparaenum}[(i)]
\item three lines $D_{010}, D_{001}, D_{011}$ not passing through the same point;\\
\item a conic $D_{100}$ passing through the two intersections points of the line $D_{011}$ with the two other lines $D_{100}, D_{010}$ and not tangent to any of the three lines. 
\end{inparaenum}

Conversely, any $\mathbb Z_2^3$--cover $f: S\longrightarrow \PP^2$ with the above branch data is related to the Del Pezzo case {\bf (4.222)}.
\end{proposition}

\begin{proof} Fix an involution $k\in \mathbb Z_2^3$ and consider the map $g: S'\longrightarrow S'/\langle k\rangle=Q\subset \PP^3$, where $Q$ is a smooth quadric in $\PP^3$. The branch curve of $g$ is easily seen to be a smooth genus 1 quartic curve $\Gamma$ cut out on $Q$ by another quadric. Fix a general point $p\in Q$ and project $Q$ to the plane from $p$. The map $g: S'\longrightarrow Q$ birationally transforms to a double cover $h: S\longrightarrow \PP^2$, with branch curve the image of $\Gamma$ in the projection, that is a quartic curve $\Gamma'$ with two distinct nodes $\xi,\eta$. Note that $S$  has two double points over $\xi$ and $\eta$ and its resolution is the blow--up of $S'$ at the two points in $g^{-1}(p)$. 

Next we consider the $\mathbb Z_2^2$--cover $r: \PP^2\longrightarrow \PP^2$
given by the action of  $\mathbb Z_2^2\cong \mathbb Z_2^3/\langle k\rangle$ on $\PP^2$
that has to globally fix $\Gamma'$. The cover $f: S\longrightarrow \PP^2$ is the composition of $h$ with $r$. 

In affine coordinates $(x,y)$ in the plane, the map $r$ is given by $(x,y)\mapsto (x^2,y^2)$ (see Remark \ref {rem:ot}). We may assume that the curve $\Gamma'$ has double points at the points at infinity of the $x$ and $y$ axes, and must be fixed by $r$, so it has to have an equation of the form
$$
ax^2y^2+bx^2+cy^2+d=0.
$$
with $a,b,c,d$ sufficiently general. The image of this curve via $r$ is the conic $\Delta$ with equation
$$
axy+bx+cy+d=0
$$  
that passes through the points at infinity of the $x$ and $y$ axes and is part of the branch curve of $f$. Moreover $r$ has the branch curve consisting of the union of the three lines given by the $x, y$ axes and the line at infinity. I denote these three lines $D_{010}, D_{001}, D_{011}$ respectively and the conic $\Delta$ by $D_{100}$ that verify (i)--(ii) of the statement of the proposition.

Conversely, suppose we have a $\mathbb Z_2^3$--cover $f: S\longrightarrow \PP^2$ with the branch data as in (i)--(ii). Then the branch curve $D=D_{010}+ D_{001}+ D_{011}+D_{100}$ has two triple points over the two points $\xi, \eta$ (at infinity of the $x$ and $y$ axes) where the conic $D_{100}$ intersects the line $D_{100}$ (at infinity). We need to resolve these singularities and to do so, we blow--up $\xi$ and $\eta$, thus getting the morphism $\pi: P\longrightarrow \PP^2$ with two exceptional divisors $E_1, E_2$ contracted to $\xi, \eta$ respectively. By pulling back the  $\mathbb Z_2^3$--cover $f: S\longrightarrow \PP^2$ via $\pi$ we have a $\mathbb Z_2^3$--cover $\phi: X\longrightarrow P$ that has as building data the pull--back on $P$ of the building data of $f: S\longrightarrow \PP^2$ via $\pi$.

Let $H$ be the pull--back to $P$ of a general line of $\PP^2$. Consider the strict transforms $D'_{010}, D'_{001}, D'_{011}, D'_{100}$ of $D_{010}, D_{001}, D_{011}, D_{100}$ on $P$. We have
$$
D'_{010}\sim H-E_1,\,\, D'_{001}\sim H-E_2,\,\, D'_{011}\sim H-E_1-E_2,\,\, D'_{100}\sim 2H-E_1-E_2.
$$
The branch data of $\phi: X\longrightarrow P$ are
$$
\bar D_{010}=D'_{010}+E_1, \,\, \bar D_{001}=D'_{001}+E_2, \,\, \bar D_{011}=D'_{011}+E_1+E_2, \,\, \bar D_{100}=D'_{100}+E_1+E_2.
$$
So  the branch curve of $\phi$ is not reduced and therefore $X$ is not normal. Let us normalize it with the recipe given in Section \ref {sss:norm}. If we do so, we get a new $\mathbb Z_2^3$--cover $\psi: Y\longrightarrow P$ with branch data
$$
\begin{array}{c}
 {\bar D}'_{010}=D'_{010}, \,\, {\bar D}'_{001}=D'_{001}, \,\, {\bar D}'_{011}=D'_{011},\\
 {\bar D}'_{100}=D'_{100},\,\, {\bar D}'_{110}=E_2,\,\, {\bar D}'_{101}=E_1, \,\, {\bar D}'_{111}=0,\\
\end{array}
$$
and $Y$ is now smooth. Accordingly the building data are
$$
\begin{array}{c}
L_{100} \sim L_{011} \sim H, \,\, L_{010} \sim H-E_1, \,\, L_{0001} \sim H-E_2,
\\
L_{110}\sim L_{101}\sim  L_{111}\sim 2H-E_1-E_2.
\end{array}
$$
Then we can compute using \eqref {eq:chi} and \eqref {eq:K2}, getting $\chi(\mathcal O_Y)=1$ and $K^2_Y=2$. 

Moreover, if $\bar D'$ is the branch curve, we have
$$
2K_Y\sim \psi^*(2K_P+\bar D')\sim \psi^*(-H)
$$
that proves that the bigenus of $Y$ is 0 and therefore also the geometric genus is 0 and  the irregularity is 0, so that $Y$ is rational. Moreover it is clear that $Y$ is a Del Pezzo surface. \end{proof}

\begin{remark}\label{rem:simpl} In the case of Proposition \ref {prop:geiserol} it is possible to simplify a bit the branch curve in the following way. The branch curve has two triple points $\xi, \eta$. Consider the intersection point $\zeta$ of the conic $D_{100}$ with the line $D_{001}$ and make a quadratic transformation based at $\xi, \eta, \zeta$. This transformation contracts to points the lines $D_{011}$ and $D_{001}$, and maps the conic $D_{100}$ to a line. Moreover it blows up the three points  $\xi, \eta, \zeta$ to three lines, that belong to the new branch curve. To be precise the new branch curve consists of five lines that are $D'_{100}$ (the image of the conic $D_{100}$), $D'_{010}$ (the image of the line $D_{010}$), $D'_{110}$ (the line corresponding to the blow--up of the point $\eta$), plus $D'_{101}$ consisting of the two lines corresponding to the blow--ups of the points 
$\xi$ and $\zeta$.

As a sanity check, I compute the building data, that are all isomorphic to $\mathcal O_{\PP^2}(1)$ except 
$$
L'_{100}\cong L'_{110}\cong L_{011}\cong \mathcal O_{\PP^2}(2)
$$
so that, using \eqref {eq:chi} and \eqref {eq:K2} we see that for the new $\mathbb Z_2^3$--cover $f': S'\longrightarrow \PP^2$ with this new branch curve $D'$ consisting of five lines, we have $\chi(\mathcal O_{S'})=1$ and $K_{S'}^2=2$ and 
$$
2K_{S'}\sim (f')^*(2K_{\PP^2}+\bar D')\sim  (f')^*(\mathcal O_{\PP^2}(-1)). 
$$
\end{remark}

\subsubsection{The degree 2 Del Pezzo case} In the case  {\bf (2.G22)}, the $\mathbb Z_2^3$--cover $f: S\longrightarrow \PP^2$ arises from a Del Pezzo surface $S'$ of degree 2. The group $\mathbb Z_2^3$ contains the Bertini involution $\sigma$.

\begin{proposition}\label{prop:geisero}  In the case {\bf (2.G22)}, the  $\mathbb Z_2^3$--cover $f: S\longrightarrow \PP^2$ has, up to a Cremona transformation of the plane, the following branch data:\\
\begin{inparaenum}[(i)]
\item three lines $D_{100}, D_{010}, D_{110}$ not passing through the same point;\\
\item an irreducible conic $D_{011}$ not passing through the intersection points of 
 the lines $D_{100}, D_{010}, D_{110}$ and not tangent to them. 
\end{inparaenum}

Conversely, any $\mathbb Z_2^3$--cover $f: S\longrightarrow \PP^2$ with the above branch data is in  the Del Pezzo case {\bf (2.G22)}.
\end{proposition}

\begin{proof} We start with the double cover $g: S'\longrightarrow \PP^2$ given by  the Bertini involution  branched along a smooth quartic curve $\Gamma$. Then $f$ is obtained by composing $g$ with the $\mathbb Z_2^2$--cover $h: \PP^2\longrightarrow \PP^2$ as in Remark \ref {rem:ot} that globally fixes the curve $\Gamma$. Here $\mathbb Z_2^2\cong \mathbb Z_2^3/\langle \sigma \rangle$. The branch curve of $h: \PP^2\longrightarrow \PP^2$ is given by three lines not in a pencil. The image of $\Gamma$ via $h$ is a conic $\Delta$ in general position with respect to the above three lines. These are the components of the branch curve of $f: S\longrightarrow \PP^2$ and one checks that (i) and (ii) are verified.

Precisely,  we need to check that $D_{011}$ does not  pass through the intersection points of 
 the lines $D_{100}, D_{010}, D_{110}$ and is not tangent to them. To see this, we can assume that the lines $D_{100}, D_{010}, D_{110}$ have respectively equation $x=0, y=0, z=0$, and that $D_{011}$ has equation $f_2(x,y,z)=0$. Recalling Theorem \ref {thm:class2}, case {\bf (2.G22)}, the equation of the smooth quartic $\Gamma$ is $F(x,y,z)=f_2(x^2,y^2,z^2)=0$. Now, suppose that $D_{011}$ passes through one of the intersection points of $D_{100}, D_{010}, D_{110}$, say through the point $(0:0:1)$. Then 
 $$
 f_2(x,y,z)=ax^2+by^2+cxy+dxz+eyz,
 $$
 hence 
 $$
 f_2(x^2,y^2,z^2)=ax^4+by^4+cx^2y^2+dx^2z^2+ey^2z^2,
 $$
 and it follows that $\Gamma$ is singular at $(0:0:1)$, a contradiction. Similarly, assume that one of lines $D_{100}, D_{010}, D_{110}$, say $x=0$, is tangent to $D_{011}$. Then the system
 $$
 f_2(x,y,z)=ax^2+by^2+cxy+dxz+eyz+fz^2=0,\,\,\, x=0
 $$
 has a solution with multiplicity 2. This is equivalent to say that $e^2-4bf=0$. Now look at the system
 $$
 \begin{array}{c}
 \frac {\partial F}{\partial x}=2x\frac {\partial f_2}{\partial x}=0
 \\
 \frac {\partial F}{\partial y}=2y\frac {\partial f_2}{\partial y}=2y(2by+cx+ez)=0
 \\
  \frac {\partial F}{\partial z}=2z\frac {\partial f_2}{\partial z}=2z(2fz+dx+ey)=0
  \end{array}
 $$
and, since $e^2-4bf=0$, this has a solution for $x=0$, meaning that $\Gamma$ is singular, a contradiction again. 

Conversely, if (i) and (ii) are verified, we have that $S$ is smooth because the branch curve $D$ has normal crossings. Moreover 
$$
2K_S\sim f^*(2K_{\PP^2}+D)\sim f^*(-H)
$$
so that  the bigenus of $S$ is 0 and $K^2_S=2$. Moreover, one easily computes the builiding data and using \eqref {eq:chi} one checks that $\chi(\mathcal O_S)=1$. So that $S$ is rational and it is in fact a Del Pezzo surface of degree 2.
\end{proof}

\subsection{The classification in the $r=4$ case} In the Del Pezzo case with $r=4$, according to Theorem  \ref {thm:class2}, we have to take care only of the case {\bf (4.2222)}. In the case  {\bf (4.2222)}, the $\mathbb Z_2^4$--cover $f: S\longrightarrow \PP^2$ arises from  a Del Pezzo surface $S'$ of degree 4 given by equations \eqref {eq:dp4} and the group $G=\mathbb Z_2^4$ is given by \eqref {eq:GGG}. 

\begin{proposition}\label{prop:r4} In the case {\bf (4.2222)}, the $\mathbb Z_2^4$--cover $f: S\longrightarrow \PP^2$ has, up to a Cremona transformation of the plane, branch data given by five lines $D_{1000},D_{0100}, D_{0010},$ $D_{0001}, D_{1111}$ in general position.

Conversely, any $\mathbb Z_2^4$--cover $f: S\longrightarrow \PP^2$ with the above branch data fall in the Del Pezzo case {\bf (4.2222)}.
\end{proposition}

\begin{proof} Let us go back to the description of case {\bf (4.2222)} in Theorem \ref {thm:class2}. The $\mathbb Z_2^4$--cover $f: S\longrightarrow \PP^2$ arises as the quotient of the Del Pezzo surface $S$ of degree 4, given by equations \eqref {eq:dp4}, by the action of the group $G\cong \mathbb Z_2^4$ of projective transformations of $\PP^4$ given by \eqref {eq:GGG}. Consider the morphism
$$
\psi: (x_1 : x_2 : x_3 : x_4 : x_5)\in \PP^4 \longrightarrow (x_1^2 : x_2^2 : x_3^2: x_4^4 : x_5^5)\in \PP^4.
$$
This is a $\mathbb Z_2^4$--cover of $\PP^4$ ramified along the five coordinate hyperplanes $x_i=0$, with $1\leq i\leq 5$. The image of the Del Pezzo surface via $\psi$ is clearly a plane and the restriction of $\psi$ to $S$ is just the $\mathbb Z_2^4$--cover $f: S\longrightarrow \PP^2$. This yields that the branch curves of $f$ are exactly five lines in general position, and it follows that the branch data are the ones as in the statement of the proposition. 

Conversely, suppose we have a $\mathbb Z_2^4$--cover $f: S\longrightarrow \PP^2$ with the above branch data. Then using \eqref {eq:data}, one computes the building data, and it turns out that
$$
L_{1110}\sim L_{1011}\sim L_{0111}\sim L_{1101}\sim L_{1111}\sim \mathcal O_{\PP^2}(2)
$$
whereas all the remaining ten building data are isomorphic to $\mathcal O_{\PP^2}(1)$. An  application of \eqref {eq:chi} and \eqref {eq:K2} shows that $\chi(\mathcal O_S)=1$ and $K_S^2=4$. Moreover 
$$
2K_S\sim f^*(2K_{\PP^2}+ D)\sim f^*(\mathcal O_{\PP^2}(-1))
$$
whence it follows that $S$ is rational and it is actually a Del Pezzo surface of degree 4. \end{proof}

  \end{document}